\documentclass[a4paper,12pt]{amsart}

\usepackage{amssymb}
\usepackage{amsmath,amsthm}
\usepackage{enumerate,paralist}
\usepackage{amstext}
\usepackage{dsfont}
\usepackage[english]{babel}
\usepackage{color}
\usepackage{textgreek}
\usepackage{hyperref}
\usepackage{comment}
\usepackage{esint}
\usepackage[normalem]{ulem}
\usepackage{cancel}

\usepackage[a4paper, left=2.5cm, right=2.5cm, top=3cm, bottom=2cm]{geometry}
\usepackage{mathtools}
\usepackage{latexsym}
\usepackage{mathrsfs}
\usepackage{MnSymbol}
\usepackage{bbm}

\theoremstyle{plain}
\newtheorem{theorem}{Theorem}

\newtheorem{corollary}{Corollary}
\newtheorem{lemma}{Lemma}
\newtheorem{proposition}{Proposition}
\theoremstyle{definition}
\newtheorem{definition}{Definition}
\newtheorem{example}{Example}
\newtheorem{remark}{Remark}

\numberwithin{theorem}{section}
\numberwithin{corollary}{section}
\numberwithin{lemma}{section}
\numberwithin{definition}{section}
\numberwithin{example}{section}
\numberwithin{remark}{section}
\numberwithin{proposition}{section}
\numberwithin{assumption}{section}

\usepackage{color}
\usepackage{xcolor}
\definecolor{MY}{rgb}{0.5,0,0.45}

\newcommand{\R}{\mathbb{R}}
\newcommand{\N}{\mathbb{N}}

\newcommand{\E}{\mathbb{E}}

\newcommand{\dD}{\mathcal{D}}
\newcommand{\eE}{\mathcal{E}}

\newcommand{\meas}{\mathfrak{m}}

\newcommand{\intM}{M^\circ}

\def\be{\beta}
\def\la{\lambda}

\def\si{\sigma}

\def\Om{\Omega}

\def\dM{\partial M}

\def\lp{\left(}
\def\rp{\right)}

\DeclareTextFontCommand{\dc}{\color{blue}}
\newcommand{\measrestr}{%
  \,\raisebox{-.127ex}{\reflectbox{\rotatebox[origin=br]{-90}{$\lnot$}}}\,%
}

\title[Dynkin condition for manifolds with boundary]{A Dynkin condition for manifolds with boundary}

\author{Marie Bormann}
\address{M.~Bormann: Institute for Applied Mathematics, Hausdorff Center for Mathematics, Universität Bonn, Endenicher Allee 60, 53115 Bonn}
\email{mbormann@uni-bonn.de}

\author{David Tewodrose}
\address{D.~Tewodrose: Department of Mathematics and Data Science, Vrije Universiteit Brussel, Pleinlaan 2, B-1050 Elsene, Belgium.}
\email{david.tewodrose@vub.be}

\begin{document}

\begin{abstract}
We propose a Dynkin-type condition for smooth Riemannian manifolds with boundary. We show that this condition implies bi-Lipschitz equivalence with a Bakry-\'{E}mery weighted Riemannian manifold obtained via a time change. 
As a consequence, we obtain various results, including a local doubling property as well as lower bounds on the Neumann spectral gap and logarithmic Sobolev constant. The local doubling property also yields a new precompactness theorem for manifolds with boundary.
\end{abstract}

\maketitle

\section{Introduction}

Over the past decade, the geometric and analytical consequences of Dynkin and Kato Ricci curvature lower bounds on geodesically complete Riemannian manifolds have attracted increasing attention. Representative contributions include the works of Rose and Stollmann \cite{RoseStollmann,RoseStollmann2}, Carron \cite{C16}, Rose \cite{Rose}, Carron and Rose \cite{CarronRose}, Rose and Wei \cite{MR4520295}, Carron, Mondello, and the second author \cite{CMT,CMT2,CMT_Torus,MR4925251,CMT5}, Lee \cite{Lee}, Impera, Rimoldi and Veronelli \cite{IRV}, Güneysu and Kuwae \cite{GuneysuKuwae}, Güneysu and Marot \cite{GM}, Mondino and Perales \cite{MondinoPerales}, among others. These conditions demand that the quantity
\begin{equation}\label{eq:Dynkin}
\sup_{x\in M} \int_0^T \int_M p(s,x,y) \mathrm{Ric}_-(y) d\mu(y) ds
\end{equation}
be bounded for some $T>0$ in the Dynkin case, and go to zero as $T\to 0$ in the Kato one. Here $p$ denotes the heat kernel of the manifold $(M,g)$, $\mathrm{Ric}_- : M \to \R_+$ is the negative part of the optimal lower bound on the Ricci curvature tensor and $\mu$ is the Riemannian volume measure. For the whole paper we only consider manifolds of dimension $n \ge 2$.

In this fruitful line of research, manifolds with boundary have received comparatively little attention. For such manifolds, which are not geodesically complete, the heat kernel depends on the choice of a boundary condition. Moreover, the function $\mathrm{Ric}_-$ alone does not capture the curvature effects induced by the boundary. We therefore consider the heat kernel associated with the Neumann boundary condition and use $\mathrm{I\!I}_- : \partial M \to \R_+$ to account for boundary effects. The latter function is defined as the negative part of the optimal lower bound on the second fundamental form $\mathrm{I\!I}^g$ of the boundary $\partial M$.

Inspired by Stollmann and Voigt \cite{MR1378151}, Güneysu \cite{MR3751359}, Erbar, Rigoni, Sturm and Tamanini  \cite{MR4403622}, Braun \cite{Braun}, and Carron, Mondello, and the second author \cite{MR4925251},
we propose the following definition, where $\si$ is the boundary measure induced by $g$ on $\dM$.

\begin{definition}[Neumann--Dynkin condition]
We say that a smooth metrically complete Riemannian manifold $(M^n,g)$ with a non-empty boundary satisfies a Neumann--Dynkin condition if there exist $T>0$ and $\gamma \in [0,1/(n-2))$ such that
\begin{equation}\label{eq:NeumannDynkin}\tag{$\mathrm{ND}_{T,\gamma}$}
k_T (\meas) := \sup_{x\in M} \int_0^T \int_{M} p^N(s,x,y)  d\meas(y) ds \le \gamma \qquad  \text{with }\meas := \mathrm{Ric}_- \mu + \mathrm{I\!I}_- \sigma.
\end{equation}
\end{definition}

Note that if $\gamma= 0$, then the positivity of $p^N$ yields $\mathrm{Ric} \ge 0$ and $\mathrm{I\!I} \ge 0$. Moreover, if $\mathrm{Ric}_-$ and $\mathrm{I\!I}_-$ admit non-negative upper bounds $K$ and $L$, respectively, then $\meas \le K \mu + L \sigma$, and hence for any $T>0$,
\begin{equation}\label{eq:boundKL}
k_T(\meas) \le (K+cL)T + c L\sqrt{T}
\end{equation}
for some constant $c>0$ depending on the manifold only (see Remark \ref{rem:bounds}). In particular, any compact Riemannian manifold with boundary satisfies a Neumann--Dynkin condition for a sufficiently small $T>0$. The same holds  in the non-compact case provided the Ricci curvature of the interior and the second fundamental form of the boundary are uniformly bounded from below.

We establish several consequences of \eqref{eq:NeumannDynkin}. Let $V_g(x,r)$ denote the volume of a geodesic ball in $M$ centered at $x$ with radius $r$, and let $\fint$ mean averaging.

\begin{theorem}[Doubling and Poincaré]\label{th:doubling&Poincaré}
    Let $(M^n,g)$ be a smooth metrically complete Riemannian manifold with a non-empty boundary satisfying \eqref{eq:NeumannDynkin} for some $T>0$ and $\gamma \in [0,1/(n-2))$. Then the following holds.
    \begin{enumerate}
        \item The manifold $(M^n,g)$ satisfies a local doubling property: there exists a constant $C_D=C_D(n,T,\gamma)>0$ such that for any $x\in M$ and $r \in (0,\sqrt{T})$,
    \begin{equation*}\label{eq:VD}
    V_g(x,r) \le C_D \, V_g(x,r/2).
    \end{equation*}
    \item The manifold $(M^n,g)$ satisfies a local $L^1$ Poincaré inequality: there exists $C_P=C_P(n,T,\gamma)>0$ such that for any $x \in M$ and $r \in (0,\sqrt{T})$, for any $f \in \mathcal{C}^1(M)$,
        \begin{equation*}\label{eq:PI1}
        \fint_{B_r(x)} \left| f - \fint_{B_r(x)} f d \mu\right| d \mu \le C_P \,  r \fint_{B_r(x)} |df| d\mu.
        \end{equation*}
    \end{enumerate}
\end{theorem}

Since the previous local doubling property involves constants that depend only on $T$ and $\gamma$, Gromov's metric compactness theorem \cite[Proposition 5.2]{Gromov} directly implies the following precompactness result.

\begin{theorem}\label{th:precompactness}
    For $T>0$ and $\gamma \in [0,1/(n-2))$, consider the collection $\mathcal{M}(n,T,\gamma)$ of isometry classes of pointed smooth metrically complete Riemannian manifolds $(M^n,g,o)$ with non-empty boundary satisfying \eqref{eq:NeumannDynkin}. Then $\mathcal{M}(n,T,\gamma)$ is precompact in the pointed Gromov--Hausdorff topology.
\end{theorem}
For uniformly locally doubling metric measure spaces that additionally satisfy a uniform local Poincaré inequality and uniform upper and lower volume bounds, a Mosco--Gromov--Hausdorff precompactness theorem is available: see the work of Carron, Mondello and the second author \cite[Theorem 1.17]{CMT}, for instance, after results from Kasue \cite[Theorem 3.4]{Kasue} and Kuwae and Shioya \cite[Theorem 5.2]{KuwaeShioya}. Applied to our context, this statement additionally provides Mosco convergence of the Dirichlet energies and convergence of the heat kernels, see Theorem \ref{th:precompactness2}.

Earlier precompactness results for manifolds with boundary include works by Kodani \cite{Kodani}, Anderson, Katsuda, Kurylev, Lassas, and Taylor \cite{AKKLT}, Wong \cite{Wong}, and Knox \cite{Knox}; we refer to the survey of Perales \cite{Peralessurvey} for a nice overview of these results. Further contributions were provided by Perales and Sormani \cite{PeralesSormani}, Perales \cite{Perales2016, Perales2020}, and Müller \cite{Muller}, among others. However, the uniform Neumann--Dynkin assumption proposed in this paper is considerably weaker than those considered in the previous literature, especially because it is not pointwise. We provide examples in this regard (Example \ref{ex:3}). As such, our work opens the way for a comprehensive study of Neumann--Dynkin limit spaces which play, for limits of manifolds with uniform pointwise boundary-geometry bounds, the role that Kato and Dynkin limit spaces play for Ricci limit spaces in the boundaryless setting.

Besides the previous statements, we also derive from \eqref{eq:NeumannDynkin} a Neumann spectral gap estimate. Recall that for a compact Riemannian manifold with boundary $(M^n,g)$, the eigenvalue problem for the Neumann Laplacian writes as
\begin{equation*}
\begin{cases}
\Delta_g u = \lambda u & \text{on $M$}, \\
\partial_\nu^g u = 0 & \text{on $\partial M$},
\end{cases}
\end{equation*}
where $\Delta_g$ is the non-negative Laplace--Beltrami operator and $\partial_\nu^gu$ is the outward normal derivative of $u$. The solutions of this problem form the Neumann spectrum, which is an unbounded non-decreasing sequence of real numbers starting with $\lambda_0(M,g)=0$. The Neumann spectral gap $\lambda_1(M,g)$ is the first positive number in this sequence. It can always be bounded from below in terms of the Neumann isoperimetric constant via Cheeger's inequality~\cite{MR402831}. However, in order to obtain explicit lower bounds, information on the geometry of both the interior and the boundary is needed in general.

Under interior Ricci curvature lower bound, an explicit lower bound on $\lambda_1(M,g)$ was obtained by Li and Yau for the case of a convex boundary \cite{MR834612} and by Chen when $\mathrm{I\!I}^g$ is bounded from below and the interior R-rolling ball condition holds \cite{MR993745}. More recently, Rose and Wei \cite{MR4520295} obtained an explicit lower bound via a Li-Yau gradient estimate under a contractive Dynkin condition (referred to as Kato condition there) on $\mathrm{Ric}_-$, a pointwise lower bound on $\mathrm{I\!I}^g$, an interior $R$-rolling ball condition and an upper bound on the diameter. In \cite{MR4523286}, Post, Ramos Olivé and Rose  provides a similar result under an $L^p$ bound on $\mathrm{Ric}_-$ and further assumptions, stronger than in~\cite{MR4520295}. Here we obtain the following.

\begin{theorem}[Neumann spectral gap]
Let $(M^n,g)$ be a compact smooth Riemannian manifold of diameter $D>0$ with a non-empty boundary. Assume there is $T>0$ such that
\begin{equation}\label{eq:D'}\tag{D'}
0< k_T(\meas) < \frac{1}{3(n-2)} \, \cdot
\end{equation}
Then there exists an explicit constant $C=C(k_T(\meas),T,D)>0$ such that
\begin{equation*}
\lambda_1(M,g) \ge C.
\end{equation*}
\end{theorem}

We also obtain an explicit upper bound on the logarithmic Sobolev constant, see Theorem \ref{thm:logSobolev}.

Our main technical tool to establish the previous results is a time-change transformation. More precisely, we prove that under \eqref{eq:NeumannDynkin} the manifold $(M^n,g)$ is bi-Lipschitz equivalent to a weighted Riemannian manifold with convex boundary satisfying a Bakry--\'Emery curvature-dimension condition, see Theorem \ref{thm:timechange}. This allows us to derive several consequences on the original space from the time-changed one. To construct the latter space, we adapt the approach developed by Carron, Mondello, and the second author in \cite{MR3751359} for geodesically complete manifolds (i.e., without boundary) to the present setting. More precisely, the time change function is obtained by solving a Schrödinger equation with potential $\mathrm{Ric}_-$ and Robin boundary condition involving $\mathrm{I\!I}_-$. To this purpose, we construct the associated semigroup and its generator from a perturbed Dirichlet form via the KLMN theorem (see e.g.~\cite[Theorem B.10]{MR3751359}) and derive bounds for the operator norm of the semigroup in $L^p, p\in[1,+\infty]$. Besides the context of curvature conditions, our construction is of independent interest for the study of Schrödinger problems with Robin boundary condition on manifolds, as it works for general functions $w$ and $v$ instead of $\mathrm{Ric}_-$ and $\mathrm{I\!I}_-$, respectively.

A time change producing boundary convexity and interior Bakry-\'{E}mery condition was previously constructed by Wang in \cite{MR2344874} via different methods. There it was used to derive an explicit lower bound on the Neumann spectral gap. However, this geometric construction required a pointwise lower bound on the Ricci curvature, a pointwise upper bound on the sectional curvature, pointwise lower and upper bounds on  $\mathrm{I\!I}^g$, and assumed that the distance between $\partial M$ and its focal cut locus is positive. In contrast, our analytic approach replaces the pointwise bounds on the Ricci/sectional curvature and the second fundamental form with a weaker Neumann–Dynkin condition, and removes any additional geometric assumption.

We conclude the introduction by briefly mentioning further related works: In \cite{boldtetal}, 
 Boldt, Güneysu and Pigola build a conformal change yielding convex boundary and bounded geometry via flatzoomers. In the non-smooth setting, Lierl and Sturm \cite{MR3742815} study convexification and preservation of the RCD condition via conformally equivalent metrics. The behaviour of synthetic curvature--dimension conditions
under time change has been investigated by 
 Sturm \cite{SturmJFA2018,MR4182834} and by Han and Sturm \cite{MR4386845}. 

\subsubsection*{Acknowledgments.} MB is funded by the Deutsche Forschungsgemeinschaft (DFG, German Research Foundation) under Germany's Excellence Strategy – EXC-2047/2 – 390685813. DT is funded by the Research Foundation – Flanders (FWO) via the Odysseus II programme no.~G0DBZ23N. He would also like to thank the Isaac Newton Institute for Mathematical Sciences, Cambridge (UK), for support and hospitality during the programme \textit{Geometric spectral theory and application} where partial work on this paper was undertaken, supported by EPSRC grant no EP/Z000580/1.
\section{Preliminaries}

Throughout this paper, we consider a connected smooth manifold $M$ with a non-empty boundary $\partial M$. We always regard $\partial M$ as a smooth embedded hypersurface of $M$ without boundary, as ensured by \cite[Theorem 5.11]{Leesmoothmanifolds}. Its second fundamental form is given by $$\mathrm{I\!I}^g(X,Y):= \langle \nabla_X^g \nu_g, Y \rangle_g$$ for all smooth vector fields $X,Y \in T \partial M$, where $\nu_g$ is the outward unit normal and $\nabla^g$ is the Levi-Civita connection. We let $\intM$ be the interior of $M$. We endow $M$ with a smooth Riemannian metric $g$ and denote by $\mu$ the associated volume measure. Assuming that the induced Riemannian distance is complete, we refer to $(M,g)$ as a metrically complete smooth Riemannian manifold with boundary.

For any $p \in [1,+\infty)$, we let $L^p(M,g)$ (resp.~$L^p(\partial M,g)$) be the space of equivalence classes of measurable functions on $M$ (resp.~$\dM$) whose $p$-th power is $\mu$-integrable (resp.~$\sigma$-integrable), with associated norm $\|\cdot\|_{p}$ (resp.~$\|\cdot\|_{p,\sigma}$).  We write $L^\infty(M)$ (resp.~$L^\infty(\dM)$) for the set of essentially bounded equivalence classes of measurable functions on $M$ (resp.~$\dM$), with associated norm $\|\cdot\|_\infty$ (resp.~$\|\cdot\|_{\infty,\sigma}$). Note that the latter spaces do not depend on $g$ while the former might. Depending on the context, we may wish to emphasize that $\|\cdot\|_{p}$ is taken with respect to $\mu$, in which case we write $\|\cdot\|_{p,\mu}$. For any $p \in [1,+\infty)$, we let
\[
\|T\|_{p,p} := \sup_{f \in L^p(M,g)\backslash \{0\}} \frac{\|Tf\|_p}{\|f\|_p}
\]
denote the operator norm of a bounded operator $T : L^p(M,g) \to L^p(M,g)$. We use the same notation for the case $p=+\infty$.

For any positive integer $k$,
we let $H^k(M,g)$ be the Sobolev space of functions in $L^2(M,g)$ whose weak derivatives of order at most $k$ are in $L^2(M,g)$ as well, with norm $\|\cdot\|_{H^k}$, and we denote by $H^{-k}(M,g)$ the dual space, with duality pairing $\langle \cdot,\cdot \rangle_{H^{-k},H^k}$.

\subsection{Neumann heat kernel.} We let $p^N : (0,+\infty) \times M \times M  \to (0,+\infty)$ be the Neumann heat kernel of $(M,g)$. This is the minimal fundamental solution of the heat equation with Neumann boundary conditions
\begin{equation}\label{eq:Neumannproblem}
\begin{cases}
\partial_t u = -\Delta_g u & \text{on $\R_+ \times M$},\\
\partial_\nu^g u = 0 & \text{on $\R_+ \times \partial M $,}
\end{cases}
\end{equation}
where $\Delta_g$ is the non-negative Laplace--Beltrami operator and $\partial_\nu^g u = g(\nabla^g u,\nu_g)$. There are various ways to establish the existence of this heat kernel. If $(M,g)$ is compact, it may be constructed as the series
\begin{equation}\label{eq:expansion}
p^N(t,x,y) = \sum_{i=1}^{+\infty} e^{-t\lambda_i(M,g)}\varphi_i(x)\varphi_i(y)
\end{equation}
where $\{\lambda_i(M,g)\}$ is the Neumann spectrum of $(M,g)$ and $\{\varphi_i\}$ is an orthonormal basis of $L^2(M,g)$ made of corresponding eigenfunctions, see e.g.~\cite[Chapters 10 \& 11]{Li}. If $(M,g)$ is non-compact, then $p^N$ may be obtained as in \cite[Theorem 3.6]{Dodziuk}.  More precisely, for any relatively compact open subset $\Omega \subset M$ with smooth boundary, consider the heat equation with mixed boundary conditions
\begin{equation}\label{eq:NeumannproblemOmegak}
\begin{cases}
\partial_t u = -\Delta_g u & \text{on $\R_+ \times \Omega$},\\
\partial_\nu^g u = 0 & \text{on $\R_+ \times (\partial \Omega \cap \partial M)$,}\\
u = 0 & \text{on $\R_+ \times (\partial \Omega \cap \intM)$.}
\end{cases}
\end{equation}
This equation admits a heat kernel $p_\Omega$ for which a series representation analogous to \eqref{eq:expansion} holds. Then the global Neumann heat kernel $p^N$ of $(M,g)$ is obtained as the unique monotone limit of the mixed boundary heat kernels $\{p_{\Omega_k}\}$  of any exhaustion of $M$ by relatively compact open sets $\{\Omega_k\}$ with smooth boundaries. It follows from this construction that $p^N$ is symmetric, smooth with respect to each variable, and that it satisfies the sub-Markovian property: for any $x \in M$ and $t>0$,
\begin{equation}\label{eq:subMarkov}
\int_M p^N(t,x,y) d\mu(y) \le 1,
\end{equation}
and the Chapman--Kolmogorov identity: for any $x,y\in M$ and $t,s>0$,
\begin{equation}\label{eq:ChapKol}
p^N(t+s,x,y) = \int_M p^N(t,x,z) p^N(s,z,y) d\mu(z).
\end{equation}
Moreover, there exists $c>0$ depending on $M$ such that for any $x \in M$ and $t>0$,
\begin{equation}\label{eq:subMarkov_boundary}
    \int_{\partial M} p^N(t,x,y) d\si(y) \le \frac{c}{\sqrt{t}} + c.
\end{equation}
This follows from \cite[formula (1.9')]{MR1076956} and a patching argument.

\begin{remark}\label{rem:bounds}
    Note that \eqref{eq:boundKL} is a direct consequence of \eqref{eq:subMarkov} and \eqref{eq:subMarkov_boundary}.
\end{remark}

\subsection{Neumann semigroup and trace operator.}
We denote by $H$ the Neumann Laplacian of $(M,g)$. This is the $L^2$-self-adjoint operator defined by
\[
\mathrm{dom}(H) = \{ u \in H^{2}(M,g) \, : \, \partial_\nu^gu=0 \,  \text{ $\sigma$-a.e.~on $\partial M$} \}, \qquad  Hu = \Delta_g u,
\]
where $\partial_\nu^gu = g(\mathrm{tr}(\nabla^g u), \nu_g)$ for any $u \in H^{2}(M,g)$, and $\mathrm{tr} : H^1(M,g) \to L^2(\partial M,g)$ is the trace operator. The classical Hille--Yosida theorem implies that $-H$ generates a strongly continuous semigroup of bounded self-adjoint operators $(e^{-tH})_{t\ge0}$ acting on $L^2(M,g)$ such that for any $u\in L^2(M,g)$,
\begin{equation*}
[0,\infty) \ni t \mapsto u(t):=e^{-tH} u
\end{equation*}
is the uniquely determined continuous path with $u(0)=u$ which is differentiable in $(0,\infty)$ and satisfies the heat equation with Neumann boundary condition \eqref{eq:Neumannproblem}.  This semigroup extends to a contractive semigroup acting on $L^p(M,g)$ for any $p \in [1,+\infty]$. Moreover, for any $t>0$, $u \in L^2(M,g)$ and $x \in M$.
\[
e^{-tH} u(x) = \int_M p^N(t,x,y) u(y) d \mu(y).
\]
The family $(e^{-tH})_{t \ge 0}$ is called the Neumann semigroup of $(M,g)$.

The trace operator $\mathrm{tr}: H^1(M,g)\to L^2(\partial M,g)$ is bounded with norm depending on $(M,g)$. It admits a formal adjoint $\mathrm{tr}^*:L^2(\partial M,g)\to H^{-1}(M,g)$ defined by: $$\langle \mathrm{tr}^* f,h\rangle_{H^{-1},H^1} = \int_{\partial M} f(x) \ [\mathrm{tr}  \, h] (x) d \sigma(x)$$
for any $(f,h) \in L^2(\partial M,g)\times H^1(M,g)$. We will need the following lemma that states that the restriction of the Neumann semigroup to the boundary $\partial M$ coincides with its conjugation with the trace operator. We use the same notation to denote the operator $e^{-tH}$ acting on $H^1(M,g)$ and its adjoint acting on $H^{-1}(M,g)$, and we write $\circ$ for the composition of operators.

\begin{lemma}\label{lem:conjugated_heat_flow}
For any $t>0$, $f \in L^2(\partial M,g)$ and $\sigma$-a.e.~$x \in \partial M$,
  \begin{equation}\label{eq:conjugated_heat_flow}
\mathrm{tr} \circ e^{-tH} \circ \mathrm{tr}^* f (x) = \int_{\dM} p^N(s,x,y) f(y) d\si(y). 
\end{equation}
\end{lemma}

\begin{proof}
For any $h\in H^1(M,g)$,
\begin{align*}
\langle e^{-tH}\circ \mathrm{tr}^* f, h\rangle_{H^{-1},H^1}& = \langle \mathrm{tr}^* f, e^{-tH} h\rangle_{H^{-1},H^1} \\
& = \int_{\partial M} f(y)\,  [\mathrm{tr}\circ e^{-tH} h] (y) d\si(y) \\
&= \int_{\partial M} f(y) \int_M p^N(s,y,x) h(x) d\mu(x) d\si(y)\\
&=\int_M h(x) \int_{\partial M} p^N(s,x,y) f(y) d\si(y) d\mu(x).
\end{align*}
Then $e^{-tH}\circ \mathrm{tr}^* f$ coincides $\mu$-a.e.~with the $L^2$ function 
\begin{equation*}
M \ni x \mapsto \int_{\partial M} p^N(s,x,y) f(y) d\si(y),
\end{equation*}
so that the continuity of $p^N$ up to the boundary implies the desired result.
\end{proof}

\subsection{Contractive Neumann--Dynkin class.} Inspired by \cite{MR1378151,MR3751359,MR4403622}, we single out the following class of Radon measures.

\begin{definition}
The contractive Neumann--Dynkin class $\dD(M,g)$ is the class of non-negative Radon measures $\alpha$ on $M$ for which there exists $T>0$ such that
\[
k_T (\alpha) := \sup_{x\in M} \int_0^T \int_M p^N(s,x,y) d\alpha(y) ds <1.
\]
\end{definition}

\begin{remark}
If $n \ge 3$ and $(M^n,g)$ satisfies \eqref{eq:NeumannDynkin} for some $T>0$ and $\gamma \in [0,1/(n-2))$, then $\meas \in \dD(M,g)$.
\end{remark}

We follow the lines of \cite[Lemmas VI.3 and VI.4]{MR3751359} to prove the next useful characterization of the contractive Neumann--Dynkin class.

\begin{lemma}
For any non-negative Radon measure $\alpha$ on $M$, the following statements are equivalent:
\begin{itemize}
\item[i)] $\alpha \in\dD(M,g)$,
\item[ii)] there exists $r>0$ such that
\begin{equation}\label{eq:chara}
A_r(\alpha) := \sup_{x\in M} \int_0^\infty e^{-rs} \int_M p^N(s,x,y) d\alpha(y) ds < 1.
\end{equation}
\end{itemize}
\end{lemma}

\begin{proof}
The statement is a consequence of the following fact: for any $r,t >0$, 
\begin{align}\label{eq:A(alpha)}
&\lp 1-e^{-rt}\rp A_r(\alpha) \le k_t(\alpha) \le e^{rt} A_r(\alpha).
\end{align}
Indeed, this implies that
\[
\begin{cases}
    A_r(\alpha) \le (1-e^{-rt})^{-1}k_t(\alpha) \stackrel{r\to +\infty}{\longrightarrow} k_t(\alpha) & \text{for any $t>0$},\\
    k_t(\alpha) \le e^{rt}A_r(\alpha)\stackrel{t\to 0}{\longrightarrow} A_r(\alpha) & \text{for any $r>0$.}
\end{cases}
\]

The first inequality in \eqref{eq:A(alpha)} is obtained via the following computation: for any $x \in M$,
\begin{align*}
 & \phantom{=} \int_0^\infty e^{-rs} \int_M p^N(s,x,y) d\alpha(y) ds & \\
&= \sum_{k=0}^\infty \int_{kt}^{(k+1)t} e^{-rs} \int_M p^N(s,x,y) d\alpha(y) ds & \\
&= \sum_{k=0}^\infty \int_{0}^{t} e^{-r(s+tk)} \int_M p^N(s+tk,x,y) d\alpha(y) ds & \\
&=\sum_{k=0}^\infty e^{-rkt} \int_M p^N(kt,x,z) \int_0^t e^{-rs} \int_M  p^N(s,z,y)d\alpha(y)ds d\mu(z) & \text{by \eqref{eq:ChapKol}}\\
&\le \lp \sum_{k=0}^{\infty} e^{-rkt} \rp \sup_{z\in M} \int_0^t e^{-rs} \int_M p^N(s,z,y) d\alpha(y) ds & \text{by \eqref{eq:subMarkov}}\\
&\le \frac{1}{1-e^{-rt}} \sup_{z\in M} \int_0^t  \int_M p^N(s,z,y) d\alpha(y) ds = \frac{k_t(\alpha)}{1-e^{-rt}} \, \cdot
\end{align*}
The second inequality is proved as follows:
\begin{align*}
e^{rt} A_r(\alpha) & = e^{rt} \sup_{x\in M} \int_0^\infty e^{-rs} \int_M p^N(s,x,y) d\alpha(y)ds\\
&\ge \sup_{x\in M} \int_0^t e^{r(t-s)} \int_M p^N(s,x,y) d\alpha(y)ds\\
&\ge \sup_{x\in M} \int_0^t \int_M p^N(s,x,y) d\alpha(y)ds = k_t(\alpha).
\end{align*}
\end{proof}

\begin{remark}
    If $\alpha = v \sigma$ for some Borel function $v : \partial M \to \mathbb{R}$, then i) and ii) are equivalent to:
\begin{itemize}
    \item[iii)] there exists $t>0$ such that
\begin{equation*}
\sup_{x\in M} \E^x\left[\int_0^t |v(B_s)| dL_s\right] <\frac{1}{2}
\end{equation*}
\end{itemize}
where $(B_s)_s$ is the reflected Brownian motion\footnote{strictly speaking, it is standard reflected Brownian motion sped up in time by a factor 2 for consistency with the remainder of this manuscript} on $M$ and $(L_s)_s$ is its local time at $\dM$. Indeed, the equivalence between i) and iii) is directly obtained from the work of Papanicolaou \cite[Proposition 1.1]{MR1076956}. This will not be used in the remainder of this work, but we mention it to clarify the connection with \cite{MR1076956} where a boundary Kato class was defined in terms of the reflected Brownian motion and its local time at the boundary.
\end{remark}

\section{Robin Schrödinger equation}
Recall that $(M,g)$ is a smooth metrically complete  Riemannian manifold with boundary. For Borel functions $w:M\to\R_+$ and $v:\dM\to\R_+$ the Schrödinger equation with potential $w$ and Robin boundary condition $v$ is
\begin{equation}\label{eq:Robinproblem1}
\begin{cases}
\Delta_g f - w f=  0&  \text{on $M$},\\
\partial_\nu^g f = vf & \text{on $\partial M $.}
\end{cases}
\end{equation}
The goal of the section is to show that this problem admits a solution provided the measure $m = w \mu + v \sigma$ belongs to the contractive Neumann--Dynkin class $\mathcal{D}(M,g)$. We begin with the following lemma which is a combined interior and boundary variant of~\cite[Lemma VI.7]{MR3751359}.

\begin{lemma}\label{lem:batuadap3}
For any $r>0,$ any Borel functions $w:M\to\R,\ v:\partial M\to\R$ and any $f\in H^1(M,g)$ one has
\begin{equation*}
\|\sqrt{|w|}f\|_{2,\mu}^2 + \|\sqrt{|v|}\mathrm{tr}(f)\|_{2,\sigma}^2\le A_r(m) \left( \|df\|_{2,\mu}^2 + r\|f\|_{2,\mu}^2 \right)
\end{equation*}
where $m:=|w|\mu+|v|\si$ and $A_r(m)$ is defined in \eqref{eq:chara}.
\end{lemma}

\begin{proof}
We can assume that $w,v$ are non-negative. Consider $L^2(M,g)\times L^2(\partial M,g)$ with scalar product $\langle (f_1^\circ,f_1^\partial),(f_2^\circ,f_2^\partial)\rangle:= \int_M f_1^\circ f_2^\circ d\mu + \int_{\partial M} f_1^\partial f_2^\partial d\si$ and induced norm $\|\cdot\|_{2,\mu\times\si}$.
In order to prove the claim it suffices to show that for any $h\in L^2(M,g)$,
\begin{equation}\label{eq:batustep1new}
\left\|\left(\widehat{w^{1/2}}\circ(H+r)^{-1/2}h,\widehat{v^{1/2}}\circ \mathrm{tr} \circ (H+r)^{-1/2}h\right)\right\|_{2,\mu\times\si}^2 \le A_r(m) \|h\|_{2,\mu}^2 
\end{equation}
where $\widehat{w^{1/2}}, \widehat{v^{1/2}}$ are the maximally defined multiplication operators induced by $w^{1/2}, v^{1/2}$, i.e.\ $\mathrm{dom}(\widehat{w^{1/2}}), \mathrm{dom}(\widehat{v^{1/2}})$ consist respectively of those $f\in L^2(M,g)$ which satisfy $w^{1/2}f\in L^2(M,g)$ and those $f\in L^2(\partial M,g)$ which satisfy $v^{1/2}f\in L^2(\partial M,g)$. Indeed, applying~\eqref{eq:batustep1new} to $h=(H+r)^{1/2}f$ with $f\in H^1(M,g)=\mathrm{dom}((H+r)^{1/2})$, we obtain
\begin{equation*}
\left\|\widehat{w^{1/2}}f\right\|_{2,\mu}^2 + \left\|\widehat{v^{1/2}} \mathrm{tr}(f)\right\|_{2,\si}^2 \le A_r(m) \|(H+r)^{1/2} f\|_{2,\mu}^2 = A_r(m) \left( \|H^{1/2} f\|_{2,\mu}^2 + r \|f\|_{2,\mu}^2 \right),
\end{equation*}
which is the desired statement. In order to prove~\eqref{eq:batustep1new}, we set $w_n:=\min(w,n)\in L^\infty (M)$, $ v_n:=\min(v,n)\in L^\infty (\partial M)$ and $m_n:=w_n \mu + v_n\si$ for any $n\in \N$. We consider
\begin{equation*}
T_n:L^2(M,g)\to L^2(M,g)\times L^2(\partial M,g), \ T_n:=\lp \widehat{w_n^{1/2}}\circ (H+r)^{-1/2},\widehat{v_n^{1/2}}\circ  \mathrm{tr}\circ(H+r)^{-1/2} \rp.
\end{equation*}
Since $A_r(m_n)\le A_r(m)$ for any $n$, the desired \eqref{eq:batustep1new} follows from 
\begin{equation}\label{eq:batustep2new}
\left\|T_n \right\|_{(2,\mu),(2,\mu\times\si)}^2 \le A_r(m_n)
\end{equation} 
and monotone convergence. Since $\widehat{w_n^{1/2}},\widehat{v_n^{1/2}}$ and $(H+r)^{-1/2}$ are bounded and self-adjoint and $\mathrm{tr}$ is bounded we obtain by general properties of adjoint operators that
\begin{equation*}
\left\|T_n \right\|_{(2,\mu),(2,\mu\times\si)}^2 = \left\|T_n \circ T_n^*\right\|_{(2,\mu\times\si),(2,\mu\times\si)}.
\end{equation*}
Here $T_n^*: L^2(M,g)\times L^2(\partial M,g) \to L^2(M,g)$ is given by
\begin{align*}
T_n^*(f^\circ,f^\partial)&= \lp\widehat{w_n^{1/2}}\circ(H+r)^{-1/2}\rp^*(f^\circ) + \lp \widehat{v_n^{1/2}} \circ \mathrm{tr}\circ(H+r)^{-1/2}\rp^*(f^\partial)\\
&=(H+r)^{-1/2}\circ\widehat{w_n^{1/2}}(f^\circ) + (H+r)^{-1/2}\circ \mathrm{tr}^* \circ \widehat{v_n^{1/2}}(f^\partial).
\end{align*}
Consider $(f_1^\circ,f_1^\partial),(f_2^\circ, f_2^\partial)\in L^2(M,g)\times L^2(\partial M,g)$. With the Laplace transform
\begin{equation}\label{eq:Laplace}
(H+r)^{-1}= \int_0^\infty e^{-rs}e^{-sH} ds,
\end{equation}
and the Cauchy-Schwarz inequality we obtain
\begin{align*}
    &\phantom{=}\,\,
    \left| \left\langle T_n \circ T_n^* (f_1^\circ,f_1^\partial), (f_2^\circ,f_2^\partial)\right\rangle \right| \\
    &= \left| \int_M w_n^{1/2} \circ (H+r)^{-1/2} \circ T_n^* (f_1^\circ,f_1^\partial) \, f_2^\circ(x) \, d\mu(x) \right. \\
    &\quad \quad  + \left. \int_{\partial M} v_n^{1/2} \circ \mathrm{tr} \circ (H+r)^{-1/2} \circ T_n^* (f_1^\circ,f_1^\partial) \, f_2^\partial(x) \, d\si(x) \right| \\
    &= \Bigl| \int_M w_n^{1/2} \circ (H+r)^{-1} \circ w_n^{1/2} \, f_1^\circ(x) \, f_2^\circ(x) \\
    &\quad \quad + w_n^{1/2} \circ (H+r)^{-1} \circ \mathrm{tr}^* \circ v_n^{1/2} \, f_1^\partial(x) \, f_2^\circ(x) \, d\mu(x) \\
    &\quad \quad + \int_{\partial M} v_n^{1/2} \circ \mathrm{tr} \circ (H+r)^{-1} \circ w_n^{1/2} \, f_1^\circ(x) \, f_2^\partial(x) \\
    &\quad \quad + v_n^{1/2} \circ \mathrm{tr} \circ (H+r)^{-1} \circ \mathrm{tr}^* \circ v_n^{1/2} \, f_1^\partial(x) \, f_2^\partial(x) \, d\si(x) \Bigr| \\
    &\le \int_0^\infty \int_M \int_M w_n^{1/2}(x) |f_1^\circ(y)| \, w_n^{1/2}(y) |f_2^\circ(x)| \, e^{-rs} p^N(s,x,y) \, d\mu(y) \, d\mu(x) \, ds \\
    &\quad \quad + \int_0^\infty \int_M \int_{\partial M} w_n^{1/2}(x) |f_1^\partial(y)| \, v_n^{1/2}(y) |f_2^\circ(x)| \, e^{-rs} p^N(s,x,y) \, d\si(y) \, d\mu(x) \, ds \\
    &\quad \quad + \int_0^\infty \int_{\partial M} \int_M v_n^{1/2}(x) |f_1^\circ(y)| \, w_n^{1/2}(y) |f_2^\partial(x)| \, e^{-rs} p^N(s,x,y) \, d\mu(y) \, d\si(x) \, ds \\
    &\quad \quad + \int_0^\infty \int_{\partial M} \int_{\partial M} v_n^{1/2}(x) |f_1^\partial(y)| \, v_n^{1/2}(y) |f_2^\partial(x)| \, e^{-rs} p^N(s,x,y) \, d\si(y) \, d\si(x) \, ds \\
    &\le \left( \iiint w_n(x) |f_1^\circ(y)|^2 e^{-rs} p^N(s,x,y) \, d\mu(y) \, d\mu(x) \, ds \right)^{1/2} \\
    &\quad \quad \quad \quad \times \left( \iiint w_n(y) |f_2^\circ(x)|^2 e^{-rs} p^N(s,x,y) \, d\mu(y) \, d\mu(x) \, ds \right)^{1/2} \\
    &\quad \quad + \left( \iiint w_n(x) |f_1^\partial(y)|^2 e^{-rs} p^N(s,x,y) \, d\si(y) \, d\mu(x) \, ds \right)^{1/2} \\
    &\quad \quad \quad \quad \times \left( \iiint v_n(y) |f_2^\circ(x)|^2 e^{-rs} p^N(s,x,y) \, d\si(y) \, d\mu(x) \, ds \right)^{1/2} \\
    &\quad \quad + \left( \iiint v_n(x) |f_1^\circ(y)|^2 e^{-rs} p^N(s,x,y) \, d\mu(y) \, d\si(x) \, ds \right)^{1/2} \\
    &\quad \quad \quad \quad \times \left( \iiint w_n(y) |f_2^\partial(x)|^2 e^{-rs} p^N(s,x,y) \, d\mu(y) \, d\si(x) \, ds \right)^{1/2} \\
    &\quad \quad + \left( \iiint v_n(x) |f_1^\partial(y)|^2 e^{-rs} p^N(s,x,y) \, d\si(y) \, d\si(x) \, ds \right)^{1/2} \\
    &\quad \quad \quad \quad \times \left( \iiint v_n(y) |f_2^\partial(x)|^2 e^{-rs} p^N(s,x,y) \, d\si(y) \, d\si(x) \, ds \right)^{1/2} \\
    &\le \left( \int_0^\infty \int_M \int_M e^{-rs} p^N(s,x,y) \, d(w_n\mu + v_n\si)(x) \, d(|f_1^\circ|^2\mu + |f_1^\partial|^2\si)(y) \right)^{1/2} \\
    &\quad \quad \times \left( \int_0^\infty \int_M \int_M e^{-rs} p^N(s,x,y) \, d(w_n\mu + v_n\si)(x) \, d(|f_2^\circ|^2\mu + |f_2^\partial|^2\si)(y) \right)^{1/2} \\
    &\le A_r(m_n) \, \|(f_1^\circ,f_1^\partial)\|_{2,\mu\times\si} \|(f_2^\circ,f_2^\partial)\|_{2,\mu\times\si}.
\end{align*}

Here the third to last inequality follows from applying the Hölder inequality to each of the four triple integrals and the second to last inequality is obtained by Cauchy-Schwarz inequality in $\R^4$.
\end{proof}

\begin{remark}
Note that the above lemma implies the following separate bounds. First, for any $r>0,$ any Borel function $w:M\to\R$ and any $f\in H^1(M,g)$ one has
\begin{equation*}
\|\sqrt{|w|}f\|_2^2 \le C_r(w) \left( \|df\|_2^2 + r\|f\|_2^2 \right)
\end{equation*}
where 
\begin{equation*}
C_r(w):= \sup_{x\in M} \int_0^\infty e^{-rs} \int_M p^N(s,x,y) |w(y)| d\mu(y) ds  \in [0,\infty].
\end{equation*}
Secondly, for any $r>0,$ any Borel function $v:\dM\to\R$ and any $f\in H^1(M,g)$ one has
\begin{equation*}
\|\sqrt{|v|}\, \mathrm{tr}(f)\|_{2,\si}^2 \le D_r(v)\left( \|df\|_{2,\mu}^2 + r\|f\|_{2,\mu}^2 \right)
\end{equation*}
where 
\begin{equation*}
D_r(v):= \sup_{x\in \partial M} \int_0^\infty e^{-rs} \int_{\dM} p^N(s,x,y) |v(y)|d\si(y) ds  \in [0,\infty].
\end{equation*}
Note also that, with identical proof, the previous lemma is true for complex-valued functions $w$ and $v$.
\end{remark}

We are now in a position to establish the following.

\begin{proposition}\label{prop:semigroup}
Let Borel functions $w:M\to\R_+$ and $v:\partial M\to \R_+$ be such that for some $T,\be>0$,
\begin{equation}\label{eq:m}
k_T(m) \le 1-e^{-\be T} \qquad \text{with $m:=w\mu+v\si$.}
\end{equation}
Then the quadratic form 
\begin{equation}\label{eq:a}
a(f,f):=\int_{M} |d f|_g^2 d\mu - \int_M w f^2 d\mu - \int_{\dM} v \mathrm{tr}(f)^2 d\si, \qquad  \mathrm{dom}(a):= H^1(M,g),\end{equation}
generates a self-adjoint operator $A$ with corresponding semigroup $(e^{-tA})_{t\ge0}$ acting on $L^p(M)$ for any $p\in[1,+\infty]$ such that for any $t\ge0$,
\begin{equation*}
\|e^{-tA}\|_{p,p} \le e^{\be(t+T)}.
\end{equation*}
\end{proposition}

\begin{proof}
\textbf{Step 1.}
We prove the existence of $A$ and its associated semigroup acting on $L^2(M,g)$. To this purpose, we show that the form $a$ is densely defined, closed and semibounded. From Lemma~\ref{lem:batuadap3} we know that for any $f\in H^1(M,g)$ and $r>0$,
\begin{equation*}
\|\sqrt{|w|}\, f\|_{2,\mu}^2 + \|\sqrt{|v|} \, \mathrm{tr}(f)\|_{2,\si}^2 \le A_{r}(m)\left(\|d f\|_{2,\mu}^2 + r \|f\|_{2,\mu}^2 \right).
\end{equation*}
It follows from~\eqref{eq:A(alpha)} and~\eqref{eq:m} that $$ A_r(m)\le \frac{k_T(m)}{1-e^{-r T}} \le \frac{1-e^{-\beta T}}{1-e^{-r T}} \, \cdot$$ Thus 
\begin{equation*}
a(f,f) \ge \int_M |d f|_g^2 d\mu - \frac{1-e^{-\be T}}{1-e^{-rT}}\left( \|d f\|_{2,\mu}^2 +  r\|f\|_{2,\mu}^2\right) = \frac{e^{-\be T}-e^{-rT}}{1-e^{-rT}}\|df\|_{2,\mu}^2 -r\frac{1-e^{-\be T}}{1-e^{-rT}}\|f\|_{2,\mu}^2
\end{equation*}
which shows semiboundedness if we choose $r>\beta$. With the same choice of $r$, we also obtain that the norm
\begin{equation*}
\|f\|_a := \lp a(f,f)+r\|f\|_{2,\mu}^2 \rp^{1/2}
\end{equation*}
is equivalent to the $H^1$ norm since
\begin{align*}
\|f\|_a^2 &= a(f,f)+r\|f\|_{2,\mu}^2 \ge \frac{e^{-\be T}-e^{-rT}}{1-e^{-rT}} \lp \|d f\|_{2,\mu}^2 + r\|f\|_{2,\mu}^2\rp,\\
\|f\|_a^2 &= a(f,f)+r\|f\|_{2,\mu}^2 \le \|d f\|_{2,\mu}^2 + r\|f\|_{2,\mu}^2,
\end{align*}
and this implies closedness. It then follows from \cite[Theorem VIII.15]{MR493419} that there exists a unique self-adjoint semibounded operator $A$ associated with $a$. 
In general, we have
\begin{equation}\label{eq:generator}
\mathrm{dom}(A):=\{f\in H^1(M,g)\ |\ \exists h\in L^2(M,g):\ a(f,g)=\int_M hg d\mu \ \forall g\in H^1(M,g)\}, \ Af = h.
\end{equation}
Notice that for $f$ in the core
\begin{equation*}
\{f\in H^2(M,g)\ |\ (\Delta_g-w)f\in L^2(M,g), \ \partial_\nu^g f = vf \text{ on } \partial M\}
\end{equation*}
we obtain via integration by parts
\begin{equation*}
Af = (\Delta_g - w)f.
\end{equation*}
By the Hille-Yosida theorem there is a corresponding strongly continuous self-adjoint semigroup of bounded operators $(e^{-tA})_{t\ge0}$ acting on $L^2(M,g)$ such that for $f\in L^2(M,g)$
\begin{equation*}
[0,\infty) \ni t \mapsto f(t):=e^{-tA} f
\end{equation*}
is the uniquely determined continuous path with $f(0)=f$ which is differentiable in $(0,\infty)$ and satisfies the abstract heat equation
\begin{equation*}
\partial_t f(t) = -A f(t).
\end{equation*}

\textbf{Step 2.} Assuming that $w$ and $v$ are bounded, we extend $(e^{-tA})_{t \ge 0}$ to a bounded semigroup acting on $L^\infty(M)$. To this aim, consider an exhaustion of $M$ by relatively compact open sets $\{\Omega_k\}$ with smooth boundaries. For any $k$, let $p_{\Om_k}$ be the heat kernel associated with the heat equation with mixed boundary condition \begin{equation}\label{eq:NeumannproblemOmega}
\begin{cases}
\partial_t u = -\Delta_g u & \text{on $\R_+ \times \Omega_k$},\\
\partial_\nu^g u = 0 & \text{on $\R_+ \times (\partial \Omega_k \cap \partial M)$,}\\
u = 0 & \text{on $\R_+ \times (\partial \Omega_k \cap \intM)$.}
\end{cases}
\end{equation}
Define $$\Psi_k:L^\infty([0,T]\times \Omega_k)\to L^\infty([0,T]\times \Omega_k)$$ by 
\begin{align*}
(\Psi_k u) (t,x) &:= \int_0^t \int_{\Omega_k} p_{\Omega_k} (t-s,x,y) w(y) u(s,y) d \mu(y)ds \\
& + \int_0^t \int_{\partial M\cap \partial \Omega_k} p_{\Omega_k}(t-s,x,y) v (y) u(s,y) d\sigma(y) ds.
\end{align*}
For any $t \in (0,T]$,
\begin{align}\label{eq:calculk}
    \|(\Psi_k 1) (t) \|_\infty
    & = \sup_{x \in \Omega_k} \left(  \int_0^t \int_{\Omega_k} p_{\Omega_k}(t-s,x,y) w(y) d\mu(y) ds \right. \\
    & \qquad \qquad \qquad \left. + \int_0^t \int_{\partial M \cap \partial \Omega_k} p_{\Omega_k}(t-s,x,y) v(y) d \sigma(y) ds \right) \nonumber \\
    & \le \sup_{x \in M} \left(  \int_0^t \int_M p^N(t-s,x,y) w(y) d\mu(y) ds \right. \nonumber \\
    & \qquad \qquad \qquad \left. + \int_0^t \int_{\partial M} p^N(t-s,x,y) v(y) d \sigma(y) ds \right) \nonumber \\
    & \le  k_T(m). \nonumber
\end{align}
Therefore, for any $u\in L^\infty([0,T]\times \Omega_k)$,
\begin{align}\label{eq:bound_Psi_k}
\sup_{t\in [0,T]} \| (\Psi_k u) (t) \|_\infty &\le \sup_{t\in[0,T]}\big( \|u(t)\|_\infty  \|(\Psi_k1)(t)\|_\infty \big) \le k_T(m) \sup_{t\in[0,T]} \|u(t)\|_\infty.
\end{align}
For $f\in L^\infty(M)$, define $$\Phi_k:L^\infty([0,T]\times \Omega_k)\to L^\infty([0,T]\times \Omega_k)$$ by
\begin{equation*}
(\Phi_k u) (t) := e^{-t\Delta_{\Omega_k}} \left. f \right|_{\Omega_k} + (\Psi_ku)(t)
\end{equation*}
where $(e^{-t\Delta_{\Omega_k}})_{t\ge0}$ is the heat semigroup generated by $p_{\Omega_k}$. For any $u \in L^\infty([0,T]\times \Omega_k)$,
\begin{align}\label{eq:unifboundPhi_k}
\sup_{t\in[0,T]} \|(\Phi_k u)(t)\|_\infty & \le \sup_{t\in[0,T]} \|e^{-t\Delta_{\Omega_k}} \left. f \right|_{\Omega_k}\|_\infty + \sup_{t\in [0,T]} \| (\Psi_k u) (t) \|_\infty \nonumber \\
&\le \|f\|_\infty + k_T(m) \sup_{t\in[0,T]} \|u(t)\|_\infty
\end{align}
where we use \eqref{eq:bound_Psi_k} and the sub-Markovian property of $(e^{-t\Delta_{\Omega_k}})_{t\ge0}$. Moreover, the map $\Phi_k$ is a contraction, because for any $u_1,u_2 \in L^\infty([0,T]\times \Omega_k)$,
\begin{align*}
\sup_{t\in[0,T]} \|(\Phi_k u_1)(t) - (\Phi_k u_2)(t)\|_\infty & = \sup_{t\in[0,T]} \|(\Psi_k u_1)(t) - (\Psi_k u_2)(t)\|_\infty\\
& \le \left( \sup_{t\in[0,T]}\|u_1(t)-u_2(t)\|_\infty \right) k_T(m) < \sup_{t\in[0,T]}\|u_1(t)-u_2(t)\|_\infty.
\end{align*}
Then the Banach fixed-point theorem ensures that $\Phi_k$ admits exactly one fixed point $\tilde{u}_k \in L^\infty([0,T]\times \Omega_k)$ that fulfils, for a.e.~$t \in [0,T]$,
\begin{align}\label{eq:fixedpointeq}
\tilde{u}_k(t) &= (\Phi_k \tilde{u}_k)(t) \notag \\\Leftrightarrow\ \tilde{u}_k(t) &=  e^{-t \Delta_{\Omega_k}}f + \int_0^t e^{-(t-s)\Delta_{\Omega_k}} (w \tilde{u}_k(s)) ds + \int_0^t e^{-(t-s)\Delta_{\Omega_k}} \mathrm{tr}^*(v \mathrm{tr}(\tilde{u}_k(s))) ds.
\end{align}
Moreover, inequality~\eqref{eq:unifboundPhi_k} implies the uniform bound $$\sup_{t\in[0,T]}\|\tilde{u}_k(t)\|_\infty \le \|f\|_\infty/(1-k_T(m))<\infty.$$
Since each $\Phi_k$ is order preserving, the sequence of functions $(\tilde{u}_k)$ is increasing and converges pointwise to some $\tilde{u}$. Letting $k \to +\infty$ in the previous inequality and in \eqref{eq:fixedpointeq}, we obtain $\tilde{u} \in L^\infty([0,T]\times M)$ with 
$$\sup_{t\in[0,T]}\|\tilde{u}(t)\|_\infty \le \|f\|_\infty/(1-k_T(m))$$
and
\begin{equation}\label{eq:duhamel}
\tilde{u}(t) =  e^{-t H} f + \int_0^t e^{-(t-s)H} (w \tilde{u}(s)) ds + \int_0^t e^{-(t-s)H} \mathrm{tr}^*(v \mathrm{tr}(\tilde{u}(s))) ds.
\end{equation}
Integrating against test functions we find that $\tilde{u}$ is a weak solution of
\begin{equation}\label{eq:robincauchyproblem}
\begin{cases}
\partial_t u = -A u \\
u(0) =f.
\end{cases}
\end{equation}  
Therefore, by setting $e^{-tA}f:=\tilde{u}(t)$, we obtain that $e^{-tA}$ is well-defined from $L^\infty(M)$ to $L^\infty(M)$, as desired.
\\

\textbf{Step 3.} Still assuming that $w$ and $v$ are bounded, we prove that for any $t >0$,
\begin{equation}\label{eq:step3}
    \|e^{-tA}\|_{\infty,\infty} \le e^{\beta(t+T)}.
\end{equation}
To this purpose, we shall use that
\[
\|e^{-tA}\|_{\infty,\infty} = \|e^{-tA}1\|_\infty.
\]
Define $$\Psi:L^\infty([0,T]\times M)\to L^\infty([0,T]\times M)$$ by
\begin{align*}
(\Psi u) (t) &:= \int_0^t e^{-(t-s)H} (w u(s)) ds + \int_0^t e^{-(t-s)H} \mathrm{tr}^*(v \mathrm{tr}(u(s))) ds
\end{align*}
where $(e^{-tH})_{t\ge0}$ is the Neumann heat semigroup acting on $L^\infty(M)$. Acting as in \eqref{eq:calculk} and \eqref{eq:bound_Psi_k}, we obtain that
\begin{align*}
    \sup_{t \in [0,T]}\| (\Psi 1) (t)\|_\infty \le  k_T(m)
\end{align*}
and for any $u\in L^\infty([0,T]\times M)$,
\begin{align*}
\sup_{t\in [0,T]} \| (\Psi u) (t) \|_\infty &\le k_T(m) \sup_{t\in[0,T]} \|u(t)\|_\infty  <\infty.
\end{align*}
The latter inequality implies that \begin{equation}\label{eq:boundpsi}
\|\Psi\|_{(\infty,T),( \infty,T)} \le k_T(m) <1
\end{equation}
where $\|\cdot\|_{(\infty,T),( \infty,T)}$ is the operator norm for bounded operators acting on $L^\infty([0,T]\times M)$. Then $(\mathrm{Id}-\Psi)$ is invertible on $L^\infty([0,T]\times M)$ with inverse $\sum_{l\ge0} \Psi^l$.
Moreover, $\Psi$ is non-negative in the sense that $$\Psi(u) \ge  0 \qquad \text{for $u\ge0$.}$$
Define $$\Phi : L^\infty([0,T]\times M) \to L^\infty([0,T]\times M)$$ as
\[
\Phi = e^{-tH} 1 + \Psi.
\]
Consider $\tilde{u} \in L^\infty([0,T]\times M)$ such that for a.e.~$t \in [0,T]$, $$\tilde{u}(t) = e^{-tA}1.$$  From the previous step it follows that the function $\tilde{u}$ is a fixed point of $\Phi$. Thus for a.e.~$t \in [0,T]$,
\[
\tilde{u}(t) = e^{-t H}1 + \Psi(\tilde{u})(t), \]
and then \[(\mathrm{Id}-\Psi)\tilde{u} (t) = e^{-t H}1.
\]
Define 
\begin{equation}\label{eq:defI}
I:=(\mathrm{Id}-\Psi)^{-1}1 \in L^\infty([0,T]\times M).
\end{equation} 
Using test functions with the equation $I = 1 + \Psi I$, we obtain that I is a solution of \eqref{eq:robincauchyproblem} with $f=1$. 
Note that $\tilde{u} = I$ if the manifold is stochastically complete. In general, it holds that
\begin{equation}\label{eq:utilde}
    \tilde{u} \le I.
\end{equation}
Indeed, if we set $\tilde{v}(t,x) = e^{-tH}1(x)$ for any $(t,x) \in [0,T] \times M$, then the sub-Markovian property of $p^N$ implies $\tilde{v} \le 1$.  The non-negativity of $\Psi$ yields then that
\begin{equation*}
\tilde{u} = (\mathrm{Id}-\Psi)^{-1} \tilde{v} = \sum_{l\ge0} \Psi^{l} \tilde{v} \le \sum_{l\ge0} \Psi^{l} 1 = (\mathrm{Id}-\Psi)^{-1} 1 = I.
\end{equation*}
The bound \eqref{eq:utilde} implies that
\begin{equation}\label{a}
\sup_{t\in[0,T]} \|e^{-tA}1\|_\infty \le \sup_{t\in[0,T]} \|I\|_\infty.
\end{equation} 
But
\begin{equation*}
I = 1 + \Psi(I)
\end{equation*}
so \eqref{eq:boundpsi} yields
\begin{equation*}
\sup_{t\in[0,T]} \|I(t)\|_\infty \le 1 + k_T(m) \sup_{t\in[0,T]} \|I(t)\|_\infty  
\end{equation*}
and then
\begin{equation}\label{eq:upperboundI}
 \sup_{t\in[0,T]} \|I\|_\infty  \le \frac{1}{1-k_T(m)} \le e^{\beta T}.
\end{equation}
Notice also that
\begin{equation*}
I=(\mathrm{Id}-\Psi)^{-1} 1 = 1 + \sum_{l\ge1} \Psi^l 1 \ge 1
\end{equation*}
since $\Psi$ is non-negative, hence
\begin{equation}\label{eq:boundsI}
1\le I \le e^{\beta T}.
\end{equation}

With \eqref{a}, this yields the desired bound \eqref{eq:step3} for $t \in [0,T]$. For $t>T$ consider $k\in\N$ such that $t\in[kT,(k+1)T]$. Then
\begin{equation}\label{eq:conclusionstep3}
\|e^{-tA}1\|_\infty = \|e^{-(kT+(t-kT))A}1\|_\infty \le \|e^{-TA}\|_{\infty, \infty}^k \|e^{-(t-kT)A}1\|_\infty \le e^{\beta(k+1)T} \le e^{\beta(t+T)}.
\end{equation}

\hfill

\textbf{Step 4.} We drop the boundedness assumption on $w$ and $v$, and reach the conclusion. For any $\ell$, set $w_\ell:=\min(w,\ell)$ and $v_\ell:=\min(v,\ell)$, and consider the densely defined, closed and semibounded symmetric form
\begin{equation*}
a_\ell(f,f):= \int_M |d f|_g^2 d\mu - \int_M w_\ell f^2 d\mu - \int_{\partial M} v_\ell \, \mathrm{tr}(f)^2 d\si,\ \quad \mathrm{dom}(a_l):=H^1(M,g).
\end{equation*}
Then we have monotone convergence of the forms
\begin{equation*}
\lim_\ell a_\ell(f,f)=a(f,f).
\end{equation*}
Due to~\cite[Theorem B.18]{MR3751359}, this implies convergence in $L^2(M,g)$
\begin{equation}\label{eq:convL2}
e^{-tA_\ell}f \to e^{-tA}f \text{ as } \ell\to\infty 
\end{equation}
for any $f\in L^2(M,g)$ and $ t\ge0$, where $A_\ell$ and $(e^{-tA_\ell})_{t \ge 0}$ are the operator and the semigroup induced by $a_\ell$. We extend each $e^{-tA_\ell}$ to $L^\infty(M)$ as in Step 2. For any $f\in L^\infty(M)$, this gives a monotone sequence $(e^{-tA_\ell}f)_\ell$ that is bounded in $L^\infty(M)$. We define $e^{-tA}f\in L^\infty(M)$ as its pointwise limit. The bound \eqref{eq:step3} established in the previous step holds for any $\ell$ and is preserved as $\ell \to +\infty$, so that $\|e^{-tA}f\|_{\infty}\le e^{\be(t+T)}\|f\|_\infty$. If $f\in L^2(M,g)\cap L^\infty(M)$ then $(e^{-tA_\ell}f)_\ell$ converges in $L^2(M,g)$ by \eqref{eq:convL2}, hence it converges almost everywhere up to taking a subsequence, and thus the $L^2$-limit has to coincide with the pointwise limit we took above. In this way, we have extended the semigroup $e^{-tA}$ from $L^2(M,g)$ to $L^\infty(M)$. Moreover, from the sequences $(w_\ell),(v_\ell)$, we construct a monotone and bounded sequence of functions $(I_\ell)$ as in~\eqref{eq:defI}, and we consider its pointwise limit $I$. By monotone convergence, $I$ still solves~ \eqref{eq:robincauchyproblem} and still satisfies~ \eqref{eq:boundsI}. Acting as in \eqref{eq:conclusionstep3}, we obtain the desired bound \eqref{eq:step3}. The bound on the operator norm for $p=1$ follows from self-adjointness and the bound for $p\in (1,+\infty)$ is then obtained from the Riesz-Thorin interpolation theorem. 
\end{proof}

As a consequence, we obtain the following existence result.

\begin{corollary}\label{cor:varphi}
Let Borel functions $w:M\to\R_+$ and $v:\partial M\to \R_+$ be such that \eqref{eq:m} holds for some $T,\be>0$. Let $A$ be the self-adjoint operator associated with the corresponding quadratic form $a$ defined in \eqref{eq:a}, as granted by the previous proposition. Then the equation
\begin{equation}\label{eq:robinproblem}
A \varphi + 2\be \varphi = 2\be
\end{equation} 
admits a weak solution $\varphi$ on $M \times \R_+$ such that $1\le\varphi \le 2e^{\be T}$ $\mu$-a.e.
\end{corollary}

\begin{proof}
Let $I \in L^\infty([0,T]\times M)$ be defined as in the proof of the previous proposition, i.e.~by \eqref{eq:defI} if $w,v$ are bounded, and by approximation otherwise. Adapting the arguments there, one can extend $I$ to a function in $L^\infty(\R_+\times M)$, that we still denote by $I$, such that for any $t >0$,
\[
1 \le I(t) \le e^{\beta(t+T)}.
\]
For $\mu$-a.e.~$x \in M$, set
\begin{equation*}
\varphi(x):= 2\be \int_0^{+\infty} e^{-2\be t} I(t,x) dt.
\end{equation*}
Then $\mu$-a.e.~on $M$,
\begin{equation*}
1=2\be \int_0^{+\infty} e^{-2\be t} dt \le \varphi \le 2\be e^{\be T}\int_0^{+\infty} e^{-\be t} dt = 2e^{\be T}.
\end{equation*}
Moreover, for any $h\in H^1(M,g)$,
\begin{align*}
 a(\varphi,h) + 2 \beta \int_M \varphi h d\mu
&= 2 \beta \int_0^{+\infty} e^{-2 \beta t} \lp a(I,h) + \int_M 2 \beta I f d\mu  \rp dt\\
&= 2 \beta \int_0^{+\infty} - \frac{d}{dt} \lp e^{-2 \beta t} \int_M I h d\mu\rp dt = 2 \beta  \int_M h d\mu 
\end{align*}
where we have used that $I$ is a weak solution of $\partial_tu  = -Au$, i.e.
\begin{equation*}
\frac{d}{dt} \int_M I h d\mu = - a(I,h).
\end{equation*}
Hence $\varphi$ is indeed a weak solution of~\eqref{eq:robinproblem}.
\end{proof}

\begin{remark}
The homogeneous version of~\eqref{eq:robinproblem} is studied in~\cite{MR1076956} via probabilistic methods.
\end{remark}

Adding mild regularity to the functions $w$ and $v$ yields the following result.

\begin{proposition}\label{prop:varphi}
    For $\alpha \in (0,1]$, let $w\in \mathcal{C}^{0,\alpha}_{\text{loc}}(M^\circ)$ and $v \in L^\infty_{\text{loc}}(\partial M)$ be non-negative functions such that \eqref{eq:m} holds for some $T,\be>0$. Then the problem
    \begin{equation}\label{eq:Robinproblembeta}
\begin{cases}
(\Delta_g - w) \varphi + 2\beta \varphi =  2 \beta&  \text{on $M$},\\
\partial_\nu^g \varphi = v \varphi & \text{on $\partial M $,}
\end{cases}
\end{equation}
admits a solution $\varphi \in \mathcal{C}^{1,\alpha}_{\text{loc}}(\overline{M}) \cap \mathcal{C}^{2,\alpha}_{\text{loc}}(M^\circ)$ such that $1 \le \varphi \le 2 e^{\beta T}$.

\end{proposition}

\begin{proof}\label{rem:regularity} It follows from classical elliptic regularity theory that the weak solution $\varphi$ provided by the previous corollary belongs to $ \mathcal{C}^{2,\alpha}_{\text{loc}}(M^\circ)$ so long as $w\in \mathcal{C}^{0,\alpha}_{\text{loc}}(M^\circ)$, see e.g.~\cite[Theorem 9.19]{MR1814364}. For the regularity up to the boundary we use~\cite[Thorem 5.54]{MR3059278}. Indeed, with a local boundary chart $(U,\phi)$ around $x \in \partial M$ and a cut-off function $\chi \in \mathcal{C}_c^\infty(U)$ such that $\chi = 1$ in a neighborhood of $x$ in $U$, we may rewrite \eqref{eq:Robinproblembeta} in coordinates as
\begin{align*}
D_i(a^{ij}D_j\tilde{\varphi} + b^i \tilde{\varphi}) + c^i D_i \tilde{\varphi} + c^0 \tilde{\varphi} &= D_i f^i + \Theta \quad \text{ in } V^\circ \\ 
a^{ij} \gamma_i D_j \tilde{\varphi} + [b\cdot \gamma] \,  \tilde{\varphi} + \beta^0 \tilde{\varphi} &= \psi + f\cdot \gamma \quad \text{ on } \partial V
\end{align*}
where $V = \phi (U)$ and
\begin{gather*}
a^{ij}=\tilde{\chi} \, g^{ij}\sqrt{|g|}, \quad b^i =0,  \quad c^i =0,  \quad c^0 = \tilde{\chi}\,  (2\beta - \tilde{w})\sqrt{|g|},\quad  f^i =0, \quad \Theta=2\beta \tilde{\chi} \sqrt{|g|}\\
b=0, \quad \beta^0 = \tilde{\chi} \sqrt{|g|} \tilde{v}, \quad \psi=0, \quad f=0.
\end{gather*}
Here $\gamma$ is the unit inner normal at the boundary and tilde means composition with $\varphi^{-1}$. All regularity assumptions for \cite[Theorem 5.54]{MR3059278} are fulfilled since $g$ is smooth. Hence we obtain $\varphi\in C^{1,\alpha}_{loc}(\overline{M})$ for any $\alpha\in(0,1)$.
\end{proof}

\begin{remark}
Under the assumptions of the previous proposition, the generator $A$ from~\eqref{eq:generator} satisfies
\begin{align}\label{eq:explicitgenerator}
\mathrm{dom}(A)&=\{f\in \mathcal{C}^{1,\alpha}_{\text{loc}}(\overline{M})\cap \mathcal{C}^{2,\alpha}_{\text{loc}}(M^\circ):\ \partial_\nu^g f = v f \text{ on }\partial M\}\\
Af &= (\Delta_g-w)f\, . \notag
\end{align}
\end{remark}

\begin{remark}
Note that Hölder continuity of $c^0$ and $\beta^0$ is not assumed in \cite[Thorem 5.54]{MR3059278}. As a consequence, the solution $\varphi$ from the previous proposition belongs to $C^{1,\alpha}_{loc}(\overline{M})$ for any $\alpha\in(0,1)$ even if $w$ belongs only to $L^\infty_{\text{loc}}(M)$.
\end{remark}

\section{Bakry-\'Emery via time change}

In this section, we establish our main technical result, namely Theorem \ref{thm:timechange}. It asserts that \eqref{eq:NeumannDynkin} implies bi-Lipschitz equivalence to a weighted Riemannian manifold satisfing a Bakry-\'Emery curvature-dimension condition.

\subsection{Bakry-\'Emery condition.} The celebrated Bakry-\'Emery curvature-dimension condition was initially introduced in the setting of Dirichlet spaces by Bakry and \'Emery in \cite{BakryEmery}, see also the work of Bakry \cite{Bakry} and the monograph by Bakry, Gentil, and Ledoux, \cite{MR3155209}. Its consequences on weighted Riemannian manifolds without boundary were explored by various authors, including Qian \cite{Qian}, Lott \cite{Lott}, Bakry and Qian \cite{BakryQian} and Wei and Wylie \cite{WeiWylie}. In this paper, we consider the condition on a weighted Riemannian manifold $(M^n,g,e^{-\psi}\mu)$ with boundary. Here $\psi$ is a $\mathcal{C}^2$ function on $M^\circ$. The triple $(M^n,g,e^{-\psi}\mu)$ admits an associated weighted Neumann Laplacian $H_\psi$ defined by 
\[
\mathrm{dom}(H_\psi) = \{ u \in H^{2}(M,g,e^{-\psi}\mu) \, : \, \partial_\nu^gu=0 \,  \text{ $\sigma$-a.e.~on $\partial M$} \}, \qquad  H_\psi u = \Delta_g u + \langle d\psi,du \rangle_g.
\]
In this context, the Bakry-\'Emery curvature-dimension condition may be written as follows, as explained in the note of Han \cite{HanUnpub}.

\begin{definition}\label{def:weightedbakryemery}
For $K\in\R$ and $N\in[1,+\infty]$ we say that a weighted Riemannian manifold $(M^n,g,e^{-\psi}\mu)$ with boundary satisfies the Bakry-\'{E}mery condition BE$(K,N)$ if

\begin{enumerate}
    \item the boundary $\partial M$ is convex, that is to say, for any smooth vector field $X\in T\partial M$,
\begin{equation*}
\mathrm{I\!I}^g(X,X) \ge 0,
\end{equation*}
\item for any $u\in \mathcal{C}_c^\infty(M^\circ)$ the following inequality holds pointwise in $M^\circ$:
\begin{equation}\label{eq:BE}
\langle d H_\psi u , du\rangle_g - \frac{1}{2} H_\psi (|du|_g^2) \ge \frac{(H_\psi u)^2}{N} + K |du|_g^2.
\end{equation}
\end{enumerate}  
\end{definition}

\begin{remark}
    It is worth stressing out that the Bakry--\'Emery condition for weighted Riemannian manifolds with boundary forces the boundary to be convex.
\end{remark}

\subsection{Time changes.} The theory of time changes in the setting of Dirichlet spaces is described in the monograph of Chen and Fukushima \cite{CF}, for instance. Here, we focus on the case of Riemannian manifolds with boundary. In this setting, a time change is a particular weighted Riemannian structure.

\begin{definition}\label{def:timechange}
Let $(M^n,g)$ be a Riemannian manifold with boundary, and $h \in \mathcal{C}^2(M^\circ)$. Then the time change of $(M,g)$ induced by $h$ is the weighted Riemannian manifold $(M,e^{2h}g,e^{2h}\mu)$.
\end{definition}

Such a time change may be understood as the composition of a conformal change of the Riemannian metric, passing from $g$ to $e^{2h}g$, followed by the weighting of the resulting volume measure $e^{nh}\mu$ by the factor $e^{(2-n)h}$. In other words,
\begin{equation}\label{eq:weighted}
(M,e^{2h}g,e^{2h}\mu) = (M,\overline{g},e^{-\psi}\mu')
\end{equation}
where $\overline{g} := e^{2h}g$, $\mu'$ is the volume measure of $\overline{g}$, and $\psi:=(n-2)h$.

To illustrate the effect of a time change, we may consider reflected Brownian motion\footnote{again, strictly speaking, sped up by a factor 2} $(B_t)_{t\ge0}$ on $M$ and its corresponding Dirichlet form in $L^2(M,g)$ defined by
\begin{equation*}
\mathrm{dom}(\eE):= H^1(M,g), \qquad \eE(f,h):=\int_M g( d f, d h) d\mu.
\end{equation*}
The time change $\overline{g}=e^{2h}g,\overline{\mu}=e^{2h}\mu$ leads to the Dirichlet form on $L^2(M,\bar{g})$ defined by
\begin{equation*}
\mathrm{dom}(\overline{\eE}):= H^1(M,g), \qquad \overline{\eE}(f,h):=\int_M g( d f, d h) d\mu.
\end{equation*}
where the only change is in the measure for the $L^2$ space. 
The stochastic process associated to the new Dirichlet form $\overline{\eE}$ is
$(B_{\tau(t)})_t$ where $\tau(t):=\inf \{s>0\ |\ A_s>t\}$ is a time reparametrization defined through the positive continuous additive functional associated with $\overline{\mu}$ via Revuz correspondence
\begin{equation*}
A_t:= \int_0^t e^{2h(B_s)} ds.
\end{equation*}
See~\cite[section 6.2]{MR2778606} for more details.
Thus the time change actually corresponds to running reflected Brownian motion on a different clock, namely in regions where $h$ is large, $A_t$ increases rapidly and accordingly $\tau$ increases slowly, i.e. the process is slowed down.

The following lemma may be proved verbatim as the boundaryless version from~\cite{MR4925251}, see also \cite{MR4386845,SturmJFA2018,MR4182834} for a similar statement in a non-smooth setting.
\begin{lemma}\label{lem:traforule}
Let $(M^n,g)$ be a smooth Riemannian manifold with boundary and $h\in \mathcal{C}^2(M^\circ)$. Set $\overline{g}:=e^{2h}g, \overline{\mu}:=e^{2h}\mu_g$ and $L:=e^{-2h}\Delta_g$. Then for any $q\in(0,+\infty]$ and $u\in \mathcal{C}_c^\infty(M^\circ)$ the inequality
\begin{equation*}
\langle d Lu, du\rangle_{\overline{g}} - \frac{1}{2}L|du|_{\overline{g}}^2 \ge \frac{(Lu)^2}{n+q} + (-\mathrm{Ric}_- + \Delta_g h - c(n,q)|dh|_g^2) e^{-2h}|du|_{\overline{g}}^2
\end{equation*}
holds pointwise in $M^\circ$ where
$c(n,q)=(n-2)(n+q-2)/q$.
\end{lemma}

As $L$ in the previous lemma is actually the weighted Laplacian on the weighted manifold $(M,e^{2h}g,e^{(2-n)h}e^{nh}\mu)$, this statement will be useful in order to establish a Bakry-\'{E}mery condition as in Definition~\ref{def:timechange}.

The next lemma describes how the second fundamental form of the boundary changes under a conformal change of metric. Its proof is a simple calculation that we omit here for brevity.

\begin{lemma}\label{lem:2ndff}
Let $(M^n,g)$ be a smooth Riemannian manifold with boundary and $h\in \mathcal{C}^1(\overline{M})$. Set $\overline{g}:=e^{2h}g$. Then for each $X\in T\partial M$ 
\begin{equation*}
\mathrm{I\!I}^{\overline{g}}(X,X)=e^h\lp \mathrm{I\!I}^g(X,X) +  |X|^2_g \, \partial_{\nu}^g h\rp.
\end{equation*}
\end{lemma}
It is worth pointing out that replacing $\mu$ with $\overline{\mu}$ makes no difference here, i.e.\ the second fundamental form of the boundary is affected only by the conformal change and not by the time change.

\subsection{Time change under Neumann--Dynkin condition.} We prove the following.

\begin{theorem}\label{thm:timechange}
Let $(M^n,g)$ be a smooth metrically complete Riemannian manifold with a non-empty boundary. Assume that \eqref{eq:NeumannDynkin} holds for some  $T>0$ and $\gamma \in [0,1/(n-2))$. Then there exist $h\in \mathcal{C}^1(\overline{M}) \cap \mathcal{C}^2(M^\circ)$, and constants $K\ge0$, $N> n$, $C>0$ depending on $n$ and $\gamma$ only, such that $0\le h \le C$ and$$\text{$(M,e^{2h}g,e^{2h}\mu)$ satisfies the $\mathrm{BE}(-K/T,N)$ condition.}$$ Moreover, if for some $T>0$,
\begin{equation}\label{eq:conditionD'}
k_T(\meas) < \frac{1}{3(n-2)} \, , \tag{D'}
\end{equation}
then we can choose $K=4k_T(\meas),\ N= n + 4(n-2)^2k_T(\meas)$ and $C=4k_T(\meas)$.
\end{theorem}

\begin{remark}
In dimension $n=2$ we do not need any restriction on $\gamma$ apart from finiteness. In fact, in this case, the time change becomes a conformal change and we obtain a metric conformal and bi-Lipschitz to $g$ with respect to which the boundary is convex and the Gauss curvature is bounded from below in the interior of the manifold. Indeed, assume there is $T>0$ such that $k_T(\meas)$ is finite. Let $K_g$ be the Gauss curvature of $g$, and $(K_g)_-$ its negative part. Set
\begin{equation*}
    w:= \frac{(K_g)_-}{2k_T(\meas)}, \quad v:=\frac{\mathrm{I\!I}_-}{2k_T(\meas)}
\end{equation*}
so that
\[
k_T(w \mu + v\sigma) = k_T\left( \frac{\meas}{2k_T(\meas)}\right)=\frac{1}{2} \, \cdot 
\]
We are then in a position to apply Proposition~\ref{prop:varphi} with $\beta = -\ln(1/2)/T$. Set
\[
h:=2k_T(\meas) \ln(\varphi),
\]
where $\varphi \in \mathcal{C}^1_{\text{loc}}(\overline{M}) \cap \mathcal{C}^2_{\text{loc}}(M^\circ)$ is given by this proposition. Then 
\begin{equation*}
    0 \le h \le 2k_T(\meas)\ln(4), \qquad \partial_\nu^g h =  \mathrm{I\!I}_-, \qquad \Delta_g h \ge (K_g)_- -2k_T(\meas)\ln(4)/T. 
    \end{equation*}
Endowed with the conformal metric $g_h:=e^{2h}g$, the manifold $M$ has a convex boundary, since Lemma \ref{lem:2ndff} implies that for $X\in T\partial M$,
\begin{align*}
\mathrm{I\!I}^{\overline{g}}(X,X) =e^h\lp \mathrm{I\!I}^g(X,X) +  |X|^2_g \, \partial_{\nu}^g h\rp \ge e^h|X|^2_g \lp  - \mathrm{I\!I}_- +   \, \partial_{\nu}^g h\rp = 0.
\end{align*}
Moreover, the transformation rule for the Gauss curvature
\begin{equation*}
    K_{g_h} = e^{-2h} (\Delta_g h + K_g) \ge e^{-2h} (\Delta_g h - (K_g)_-)
\end{equation*}
implies that the
Gauss curvature $K_{g_h}$ is bounded from below by $-2k_T(\meas)\ln(4)/T$. Thus in this case no restriction on $k_T(\meas)$ is needed besides finiteness, and the Bakry-Émery condition is fulfilled with $N=2$.
\end{remark}

Let us prove Theorem \ref{thm:timechange}.

\begin{proof}
We first assume that assumption~\eqref{eq:NeumannDynkin} is fulfilled. If $\gamma=0$ then no time change is needed. Else set $\lambda:=\frac{1}{2}(n-2+1/\gamma)>n-2>0$, and $w=\la \mathrm{Ric}_-$ and $v=\la \mathrm{I\!I}_-$.
For
\begin{equation*}
\be := -\frac{1}{T}\, \ln\lp\frac{1}{2}(1-(n-2)\gamma)\rp>0
\end{equation*}
we have $k_T(w\mu+v\si) \le 1-e^{-\be T}$. Then the function $\varphi \in \mathcal{C}^2_{\text{loc}}(M^\circ)\cap \mathcal{C}^1_{\text{loc}}(\overline{M})$ constructed in Proposition~\ref{prop:varphi} fulfils 
\[1\le\varphi \le 2e^{\be T}, \qquad \partial_\nu^g \varphi  = \lambda \mathrm{I\!I}_-\varphi, \qquad \Delta_g \varphi - \lambda \mathrm{Ric}_-\varphi = -2\be \varphi +2\be \ge -2\be \varphi.\]
We define
\begin{equation*}
\overline{g} := e^{2h}g, \quad \overline{\mu}:=e^{2h}\mu_g, \quad  L:=e^{-2h}\Delta_g\quad \text{where} \quad h:=\frac{1}{\lambda}\log(\varphi).
\end{equation*}
Then $(M,\overline{g})$ has a convex boundary, since Lemma~\ref{lem:2ndff} implies that for any $X \in T\partial M$,
\begin{align*}
\mathrm{I\!I}^{\overline{g}}(X,X) &= e^{h} \lp \mathrm{I\!I}^{g}(X,X) + |X|^2_g \, \partial_{\nu}^g h \rp\ge e^h|X|_g^2(-\mathrm{I\!I}_{-} + \mathrm{I\!I}_{-}) \ge 0.
\end{align*}
We consider $q\in(0,+\infty]$ such that $\la=c(n,q)$, i.e.\ $q=\frac{(n-2)^2}{(\la -(n-2))}$. Then due to Lemma~\ref{lem:traforule}, for any $u \in \mathcal{C}_c^\infty(M^\circ)$ we obtain that the following inequality holds pointwise in $M^\circ$:
\begin{equation*}
\langle d Lu, du\rangle_{\overline{g}} - \frac{1}{2}L|du|_{\overline{g}}^2 \ge \frac{(Lu)^2}{n+q} + (-\mathrm{Ric}_- + \Delta_g h - \lambda |dh|_g^2) e^{-2h}|du|_{\overline{g}}^2.
\end{equation*}
Since
\begin{equation*}
\Delta_g h = \frac{\Delta_g \varphi}{\lambda \varphi} + \la |dh|_g^2
\end{equation*}
we get
\begin{align*}
 (-\mathrm{Ric}_- + \Delta_g h - \lambda|dh|_g^2) e^{-2h}|du|_{\overline{g}}^2 & = \lp \frac{\Delta_g \varphi}{\la \varphi} - \mathrm{Ric}_- \rp e^{-2h} |du|_{\overline{g}}^2\\
 & \ge -\frac{2\be}{\la} e^{-2h} |du|_{\overline{g}}^2 \ge - \frac{2\be}{\la} |du|_{\overline{g}}^2
\end{align*}
where we have used that $\varphi \ge 1$ implies $e^{-2h}\le 1$ to get the last inequality. Overall,
\begin{align*}
\langle d Lu, du\rangle_{\overline{g}} - \frac{1}{2}L|du|_{\overline{g}}^2 \ge \frac{(Lu)^2}{n+q} -\frac{2\be}{\la} |du|_{\overline{g}}^2.
\end{align*}
Therefore, BE$(-K/T,N)$ is fulfilled on the weighted Riemannian manifold with boundary $(M,\overline{g},\overline{\mu})$, with 
$K=-\frac{4\ln(\frac{1}{2}(1-(n-2)\gamma))}{(n-2+1/\gamma)}$ and $N=n+q$. Furthermore, 
\begin{equation*}
    0\le h \le \frac{1}{n-2} \ln\lp\frac{4}{1-(n-2)\gamma}\rp.
\end{equation*}
\vspace{0.2cm}\\

We now assume that condition~\eqref{eq:conditionD'} holds.
The proof goes as above, but we choose different values for the parameters $\be$ and $\la$, namely
\begin{equation*}
\la:=\frac{1-e^{-1}}{k_T(\meas)} \text { and } \be:=\frac{1}{T} \cdot 
\end{equation*}
We again set $h:=\ln(\varphi)/\la$.
Choosing $q>0$ such that $\la=c(n,q)$ results in
\begin{equation*}
    q= \frac{(n-2)^2k_T(\meas)}{1-e^{-1}-(n-2)k_T(\meas)} \, \cdot 
\end{equation*}
We obtain that BE$(-K/T,N)$ is fulfilled with
\begin{align*}
    K&=\frac{2\be T}{\la} = \frac{2k_T(\meas)}{1-e^{-1}} \le 4k_T(\meas) \qquad \text{and} \qquad 
    N = n+q \le n+ 4(n-2)^2k_T(\meas).
\end{align*}
Moreover,
\begin{equation*}
0 \le h \le \frac{\ln(2e)}{\la} \le \frac{2}{\la} \le 4k_T(\meas).
\end{equation*}
\end{proof}

\begin{remark}
If $\mathrm{I\!I}_-\equiv0$, then the Robin boundary condition in~\eqref{eq:explicitgenerator} becomes a Neumann one, and if $\mathrm{Ric}_-\equiv0$, then the operator $A$ becomes the Laplacian without perturbation. If $\mathrm{I\!I}_-\equiv \mathrm{Ric}_- \equiv 0$, then $I\equiv1$ in the proof of Proposition~\ref{prop:semigroup} and $\varphi\equiv1$ in Proposition~\ref{prop:varphi}, so that $e^{2h}\equiv 1$ and we do not make any time change. However, if either  $\mathrm{I\!I}_-\neq0$ or $\mathrm{Ric}_-\neq0$, then we do make a nontrivial time change. In particular, in a region of $M^\circ$ where $\mathrm{Ric}_-$ is large the factor $e^{2h}$ will be large, and in a region of $\partial M$ where $\mathrm{I\!I}_-$ is large the factor $e^{2h}$ will be large near that region. Coming back to the interpretation of time changes via Brownian motion, we see that, in regions where $\mathrm{Ric}_-$ is large and reflected Brownian motion tends to spread out fast, it is slowed down by the time change we provide.
\end{remark}

\section{Doubling, Poincaré and limit spaces}

In this section, we prove that the Neumann--Dynkin condition \eqref{eq:NeumannDynkin} implies a local doubling property and a local $L^1$ Poincaré inequality with parameters depending on $T$ and $\gamma$ only. To this aim, we begin with a preliminary result. For a weighted Riemannian manifold $(M^n,g,e^{-\psi}\mu)$ with
potential $\psi\in \mathcal{C}^1(\overline{M})\cap\mathcal{C}^2(M^\circ)$, let us set
\[
V_{g,\psi}(x,r) = \int_{B_r(x)} e^{-\psi} d\mu
\]
for any $x \in M$ and $r >0$. Then the following holds.

\begin{proposition}\label{prop:BEdoubling}
    Let $(M^n,g,e^{-\psi}\mu)$ be a weighted Riemannian manifold with convex boundary satisfying the Bakry-\'{E}mery condition $\mathrm{BE}(-K,N)$ for some $K \ge 0$ and $N\in[2,+\infty)$. Then the following holds. 
    \begin{enumerate}
        \item The space $(M^n,g,e^{-\psi}\mu)$ satisfies a local doubling property: for any $R>0$ there exists $C_D=C_D(K,N,R)>0$ such that for any $x \in M$ and $r \in (0,R)$,
        \[
        V_{g,\psi}(x,r) \le C_D \,  V_{g,\psi}(x,r/2).
        \]
        \item The space $(M^n,g,e^{-\psi}\mu)$ satisfies a local $L^1$ Poincaré inequality: for any $R>0$ there exists $C_P=C_P(K,N,R)>0$ such that for any $x \in M$ and $r \in (0,R)$, for any $f \in \mathcal{C}^1(M)$,
        \[
        \fint_{B_r(x)} \left| f - \fint_{B_r(x)} f e^{-\psi} d \mu\right| e^{-\psi} d \mu \le C_P \,  r \fint_{B_r(x)} |df|_g e^{-\psi} d\mu.
        \]
    \end{enumerate}
\end{proposition}

\begin{proof}
    The Bakry-\'{E}mery condition $\mathrm{BE}(-K,N)$ implies that the weighted Riemannian manifold with convex boundary $(M^n,g,e^{-\psi}\mu)$ is an $\mathrm{RCD}^*(-K,N)$ space, as can be seen from \cite[Chapter 12]{AMS}, for instance. An $\mathrm{RCD}^*(-K,N)$ space satisfies a local doubling property and a local $L^1$ Poincaré inequality as above. Indeed, the local doubling property is a direct consequence of the Bishop--Gromov inequality for the larger class of $\mathrm{CD}(-K,N)$, see e.g.~\cite[Theorem 30.11]{Villani}, and the local $L^1$ Poincaré inequality on $\mathrm{CD}(-K,N)$ was proved by Rajala in \cite{Rajala}.
\end{proof}

\begin{remark}
The local doubling property may also be established by a suitable adaptation of the proofs of \cite{Qian,Lott,BakryQian,WeiWylie} in the presence of a convex boundary. These are based on mean comparison arguments that go through under convexity of the boundary. 
\end{remark}

We now couple Theorem \ref{thm:timechange} and Proposition \ref{prop:BEdoubling} to establish Theorem \ref{th:doubling&Poincaré}.

\begin{proof}
    By Theorem \ref{thm:timechange}, there exist $h\in \mathcal{C}^1(\overline{M})\cap \mathcal{C}^2(M^\circ)$ and $K\ge0$, $N \ge n$, $C>0$ depending on $n$ and $\gamma$ only, such that $0\le h \le C$ and $(M,g',\overline{\mu})$  satisfies the BE$(-K/T,N)$ condition, where $g' = e^{2h} g$ and $\overline{\mu}=e^{2h}\mu$. Let $\mu'=e^{nh}\mu$ be the volume measure associated with $g'$. Then $\overline{\mu}=e^{-\psi} \mu'$ with $\psi=(n-2)h$, so that $(M,g',\overline{\mu})=(M,g',e^{-\psi}\mu')$. By Proposition \ref{prop:BEdoubling}, there exists $C_D>0$ depending on $-K/T$ and $N$ only, i.e.~on $n,\gamma,T$ only, such that for any $x\in M$ and $r \in (0,e^C\sqrt{T})$,
    \[
        V_{g',\psi}(x,r) \le C_D \,  V_{g',\psi}(x,r/2).
        \]
    From $0\le h \le C$ we obtain that
    \[
    g \le g' \le e^{2C} g \qquad \text{and} \qquad \mu \le \overline{\mu} \le e^{2C} \mu.
    \]
    Use $B'$ to denote balls with respect to $g'$. Then for any $x \in M$ and $\tau>0$,
    \[
    B_\tau'(x) \subset B_\tau(x) \subset  B_{e^C\tau}'(x),
    \]
     and then
    \[
    V_g(x,\tau) \le V_{g',\psi}(x,e^C \tau) \qquad \text{and} \qquad V_{g',\psi}(x,\tau) \le e^{2C}V_g(x,\tau).
    \]
    Therefore, for any $r \in (0,\sqrt{T})$,
    \begin{align*}
        V_g(x,r) \le   V_{g',\psi}(x,e^Cr) & \le C_D \left(\frac{e^C r}{r/2}\right)^{\log_2 C_D} V_{g',\psi}(x,r/2) \le e^{(2+\log_2 C_D)C} C_D^2 \,V_g(x,r/2),
    \end{align*}
    where we use a classical consequence of the doubling property to get the second inequality. Along similar lines, one can show that a weak local $L^1$ Poincaré inequality holds: 
   there exists $\overline{C}_P=\overline{C}_P(n,T,\gamma)>0$ and $\Lambda = \Lambda(n,T,\gamma)\ge 1$ such that for any $x \in M$ and $r \in (0,\sqrt{T})$, for any $f \in \mathcal{C}^1(M)$,
        \begin{equation*}\label{eq:PI2}
        \fint_{B_r(x)} \left| f - \fint_{B_r(x)} f d \mu\right| d \mu \le \overline{C}_P \,  r \fint_{B_{\Lambda r}(x)} |df| d\mu.
        \end{equation*} 
    This can be turned into the desired local $L^1$ Poincaré inequality by applying techniques from the works of Jerison \cite{Jerison} and Maheux and Saloff-Coste \cite{MSC}.
\end{proof}

In view of Theorem \ref{th:precompactness}, we propose the following definition. Recall that  $\mathcal{M}(n,T,\gamma)$ is the set of (isometry classes of) pointed smooth metrically complete Riemannian manifolds $(M^n,g,o)$ with non-empty boundary satisfying \eqref{eq:NeumannDynkin}.

\begin{definition}
    We say that a pointed metric space $(X,d,o)$ is a Neumann--Dynkin limit if it arises as the pointed Gromov--Hausdorff limit of a sequence of elements in $\mathcal{M}(n,T,\gamma)$ for some $T>0$ and $\gamma \in [0,1/(n-2))$.
\end{definition}

Let us give a few relevant examples of Neumann--Dynkin limits.

\begin{example}
    Any bounded closed convex set $\mathcal{K} \subset \mathbb{R}^n$ with a non-empty interior can be approximated in the Hausdorff distance by closed convex sets with smooth boundary, see e.g.~\cite{Gruber}. If we endow these approximating sets with the restriction of the ambient Euclidean metric, then we obtain that they are flat Riemannian manifolds with convex boundary. Hence  $\mathcal{K}$ belongs to the Gromov--Hausdorff closure of $\mathcal{M}(n,T,0)$ for any $T>0$. Note that  $\mathcal{K} \notin \mathcal{M}(n,T,0)$ unless its boundary is smooth.
\end{example}

\begin{example}
Consider the  cylinders $\mathcal{C}_n:= [-n,n]\times \mathbb{S}^1$, $n \in \mathbb{N}$, endowed with the product Riemannian metric $g=dt^{2} + g_{\mathbb{S}^1}$. Each surface $(\mathcal{C}_n,g)$ has a convex boundary and a Gaussian curvature constantly equal to $0$ in the interior, thus $(\mathcal{C}_n,g) \in \mathcal{M}(2,T,0)$ for any $T>0$. Set $o=(0,\theta)$ for some $\theta \in \mathbb{S}^1$. Then the pointed Riemannian manifolds with boundary $(\mathcal{C}_{n},g,o)$ converge to the pointed infinite cylinder $(\mathbb{R} \times \mathbb{S}^1, g,o)$ in the pointed Gromov--Hausdorff topology. This example illustrates that a sequence satisfying a uniform Neumann--Dynkin condition may converge to a boundaryless manifold.
\end{example}

\begin{example}\label{ex:3}
Consider the cylinder $\mathcal{C}_{\frac{5\pi}{12}} = [-\frac{5\pi}{12},\frac{5\pi}{12}] \times \mathbb{S}^1$ with the Lipschitz metrically complete Riemannian metric
\[
\overline{g} = dt^2 + \cos^2\left( |t|- \frac{\pi}{6} \right) d\theta^2
\]
where $(t,\theta)$ are Fermi coordinates  along the middle circle $\mathcal{E}=\{0\} \times \mathbb{S}^1$.  Due to the $|t|$ factor in the definition of $\overline{g}$, this metric is smooth outside $\mathcal{E}$ but not differentiable in $\mathcal{E}$. The space $(\mathcal{C}_{\frac{5\pi}{12}},\overline{g})$ describes the non-smooth gluing of two bands of the unit sphere, one lying between latitudes $-30^\circ$ and $45^\circ$, the other one between $-45^\circ$ and $30^\circ$, along the boundary circles at $\pm30^\circ$. This space is a truncated variant of the gluing of two spheres given in \cite[Section 4, case $\alpha = \pi/6$]{CMT5}.

For a smooth positive function $J$ on $[-\frac{5\pi}{12},\frac{5\pi}{12}]$, the Gaussian curvature of the warped metric $dt^2 + J^2(t)d\theta^2$ is given by $K(t,\theta)=-J''(t)/J(t)$. If $J>0$ is not smooth but locally Lipschitz with $J' \in \mathrm{BV}_{\text{loc}}$, the latter formula remains true as a distributional equality. Plugging $J(t)=\cos(|t|-\pi/6)$ in the previous, we obtain that $\overline{g}$ admits a Gaussian curvature measure given by
    \[
\overline{\omega} = \overline{\mu} - \frac{2}{\sqrt{3}} \overline{\mathcal{H}}^1 \measrestr \mathcal{E}
    \]
    where $\overline{\mu}$ is the Riemannian volume measure of $\overline{g}$ on $\mathcal{C}_{\frac{5\pi}{12}} \backslash \mathcal{E}$ and $\overline{\mathcal{H}}^1 \measrestr \mathcal{E}$ is the associated one-dimensional Hausdorff measure restricted to the equator $\mathcal{E}$. The reparametrization $$t = \varphi(r) = \mathrm{sgn}(r) \left( 2 \arctan\left( \frac{2 \tan(|r|/2)-1}{\sqrt{3}} \right) + \frac{\pi}{3}\right)$$
    provides an isometry $(r,\theta) \mapsto (\varphi(r),\theta)$ from $(\mathcal{C}_{r_0},e^{2\overline{u}}g)$ onto $(\mathcal{C}_{\frac{5\pi}{12}},\overline{g})$, where 
    $$
    \mathcal{C}_{r_0} = [-r_0,r_0] \times \mathbb{S}^1 \quad \text{with} \quad r_0=2 \arctan(\sqrt{3}(1+\sqrt{2}))-\pi/2,
    $$
    $$
    e^{2\overline{u}(r)} = \frac{3}{(2-\sin(|r|))^2} \qquad \text{and} \qquad g = dr^2 + \cos^2(r) d \theta^2
    $$ where $(r,\theta)$ are reparametrized Fermi coordinates along $\mathcal{E}$. Adapting the arguments in \cite{CMT5}, one can approximate $\overline{u}$ by means of smooth Neumann heat kernel regularizations to obtain smooth Riemannian metrics $\{h_\ell\}$ on $\mathcal{C}_{r_0}$ such that:
    \begin{enumerate}
        \item $(\mathcal{C}_{r_0}, h_{\ell}) \to (\mathcal{C}_{r_0},e^{2\overline{u}}g)$  in the Gromov--Hausdorff sense as $\ell \to \infty$,
        \item the negative parts of the interior Gauss curvatures of the $h_\ell$ lie in a uniform Kato class,
        \item the negative parts of the second fundamental forms of $\partial \mathcal{C}_{r_0}$ with respect to $h_\ell$ lie in a uniform Kato class as well.
    \end{enumerate}
\end{example}

\begin{example}\label{ex:4} This example is a collapsed one, that is, the dimension of the limit is strictly less than the one of the approximating manifolds. For any positive integer $k$, consider the epigraph
    \[
    M_k := \{(x,y,z) \in \R^3 : z \ge f_k(x,y)\}
\qquad \text{where} \qquad 
    f_k(x,y) = -\frac{x^2}{2} + 2^{k-1} y^2 \, \cdot
    \]
    This is a smooth three-dimensional manifold with boundary $$\partial M_k =  \{(x,y,z) \in \R^3 : z = f_k(x,y)\}.$$ Endow $M_k$ with the flat Euclidean metric $g_e$. The second fundamental form of $\partial M_k$ may be described via the Hessian matrix
    \[
    \nabla^2f_k(x,y) = \begin{pmatrix}-1 & 0 \\ 0 & 2^k \end{pmatrix}.
    \]
    Then $\mathrm{I\!I}_-$ is constantly equal to $1$ along the sequence $\{(M_k,g_e)\}_k$, so the sequence belongs to $\mathcal{M}(3,T,\gamma)$ for some $T>0$ and $\gamma \in (0,1)$. One can easily show that
    \[
    (M_k,g_e,0_3) \to (M_\infty,g_e,0_2)
    \]
    in the pointed Gromov--Hausdorff sense,
    where $0_2$ and $0_3$ are the origins of $\R^2$ and $\R^3$, respectively, and
    \[
    M_\infty := \{(x,y) \in \R^2 : y \ge f_\infty(x)\} \qquad \text{where} \qquad 
    f_\infty(x) = -\frac{x^2}{2} \, \cdot 
    \]
\end{example}

Let us conclude this section with another precompactness result. The local $L^1$ Poincaré inequality implies the $L^2$ one: this is explained in the book of Björn and Björn \cite[Theorem 4.21]{BB}, for instance. As a consequence, we can apply \cite[Theorem 1.17]{CMT} to get the following result.

\begin{theorem}[Mosco--Gromov--Hausdorff precompactness]\label{th:precompactness2}
    For $T>0$, $\gamma \in [0,1/(n-2))$ and $\eta >0$, consider the class $\mathcal{M}(n,T,\gamma,\eta)$ of isometry classes of pointed smooth metrically complete Riemannian manifolds $(M^n,g,o)$ satisfying \eqref{eq:NeumannDynkin} and
\[
\eta^{-1} \le V_g(o,\sqrt{T}) \le \eta.
\]
Then $\mathcal{M}(n,T,\gamma,\eta)$ is precompact in the pointed Mosco--Gromov--Hausdorff topology. Moreover, if a sequence $\{(M^n_i,g_i,o_i)\} \subset \mathcal{M}(n,T,\gamma,\eta)$ converges in the pointed Mosco--Gromov--Hausdorff tomology to a space $(X,d,\mu,o,\mathcal{E})$, then the following holds.
\begin{enumerate}
    \item The Dirichlet form $\mathcal{E}$ is strongly local, regular, and admits a heat kernel $p$.
    \item The pseudo-distance $d_{\mathcal{E}}$ associated with $\mathcal{E}$ is a distance that satisfies $c d_{\mathcal{E}} \le d\le d_{\mathcal{E}}$ for some $c \in (0,1]$ depending only on $n,T,\gamma$.
    \item Let $p_i$ be the heat kernel of $(M^n_i,g_i)$ for any $i$. Then for any $t>0$, the functions $p_i(t,\cdot,\cdot)$ converge to $p(t,\cdot, \cdot)$ in the uniform Gromov--Hausdorff topology over compact sets.
\end{enumerate}

\end{theorem}

\section{Functional inequalities}

In this section, we additionally assume that $M$ is compact.

\subsection{Poincar\'{e} inequality}
We aim at bounding the Neumann spectral gap $\lambda_1=\lambda_1(M,g)$ from below in terms of the Neumann--Dynkin bound, diameter and dimension. This is equivalent to bounding the Poincar\'{e} constant from above. \\
The time change featuring a function $h$ constructed in Theorem~\ref{thm:timechange} allows to compare $\lambda_1$ with the spectral gap of the Neumann Laplacian on the time-changed manifold, which has convex boundary and fulfils a Bakry-\'{E}mery condition. Lower bounds on the spectral gap in this latter setting have been studied before by Wang~\cite{MR1262968}, Chen and Wang~\cite{MR1450586} and Bakry and Qian~\cite{MR1789850} via coupling methods resp.\ reduction to a one-dimensional Sturm-Liouville problem. In particular, we are only interested in the case $h\neq0$, else we already have non-negative Ricci curvature and convex boundary and our ansatz does not give any new results in that setting.\\
The case of a nonconvex boundary has also been investigated by Wang in~\cite{MR2158015,MR2344874} via gradient estimates for the Neumann heat semigroup resp.\ via a time change. Thus our ansatz is analogous to~\cite{MR2344874}, but two relevant distinctions are that firstly we use the contractive Dynkin bounds on Ricci curvature and second fundamental form instead of pointwise lower bounds and secondly our time change is different from the one used in~\cite{MR2344874} which is based on an explicit geometric construction using the distance-to-the-boundary function under further stronger geometric assumptions.

\begin{theorem}\label{thm:spectralgap}
Let $(M^n,g)$ be a compact smooth Riemannian manifold of diameter $D$ with a non-empty boundary. Assume there exists $T>0$ such that \eqref{eq:D'} holds. Then 
\begin{equation*}
\lambda_1 \ge \begin{cases}
\frac{\pi^2 K}{8T(\exp(Ke^{4K}D^2/(8T))-1)} & \text{if $n>2$},\\
\frac{\pi^2}{e^{4K}D^2} \sqrt{\frac{1+2e^{4K}D^2 K}{\pi^4 T}} \cosh^{-1}\lp\frac{e^{4K}D}{2}\sqrt{\frac{K}{T}}\rp & \text{if $n=2$.}
\end{cases}
\end{equation*}
Here $K=4k_T(\meas)$ is the constant from Theorem~\ref{thm:timechange}.
\end{theorem}

\begin{proof}
Let $\mu$ be the normalised volume measure on M. Let $f$ be an eigenfunction for $\lambda_1$ and assume without loss of generality that $\int_M f d\mu =0$ and $\int_M f^2 d\mu =1$. Moreover let $\overline{\lambda}_1$ be the first nontrivial eigenvalue of the weighted Neumann Laplacian on $(M, e^{2h}g,e^{2h}\mu)$ which we denote by $\overline{\Delta_g}$. For brevity in the following we mark by $\overline{\phantom{n}}$ all quantities that use the time changed metric and measure. Then
\begin{align*}
\lambda_1 &= \int_M |\nabla f|^2 d\mu = \int_M  \overline{\langle \overline{\nabla} f, \overline{\nabla} f \rangle} d\overline{\mu} \ge \overline{\lambda}_1 \lp \int_M f^2 d\overline{\mu} - \frac{\lp \int_M f d\overline{\mu}\rp^2}{\overline{\mu}(M)}\rp = \overline{\lambda}_1 \inf_{c\in \R} \int_M (f-c)^2 d\overline{\mu}\\
&\ge \overline{\lambda}_1 \inf_{x\in M} e^{2h(x)} \cdot  \inf_{c\in \R}\int_M (f-c)^2 d\mu = \overline{\lambda}_1 \inf_{x\in M} e^{2h(x)}.
\end{align*}
It now remains to bound both $e^{2h}$ and $\overline{\lambda}_1$ from below.
Without further assumptions on $\mathrm{Ric}_-$ and $\mathrm{I\!I}_-$ we know from the construction of $h$ in the proof of Theorem~\ref{thm:timechange} that $e^{2h}\ge1$.\\
Bounds for $\overline{\lambda}_1$ are provided e.g.\ in~\cite{MR1262968,MR1450586,MR1789850}.
In the following we will use the bound provided in~\cite[Theorem 1, iv)]{MR1450586}. For this purpose we write $\overline{\Delta_g}$ as a Laplacian with drift corresponding to weighting the measure by $e^{-(n-2)h}$ in the conformally changed (as opposed to time changed!) manifold $(M,g',\mu'):=(M,e^{2h}g,e^{nh}\mu)$, all quantities that use the conformally changed metric and measure are marked by '.
\begin{equation*}
\overline{\Delta_g} = e^{-2h}\Delta_g= e^{-2h}\lp \Delta_g f + (n-2)\nabla h - (n-2)\nabla h \rp= \Delta'_g - (n-2)\nabla'h .
\end{equation*}
$h$ has been constructed such that on the manifold $(M,e^{2h}g,e^{2h}\mu)$ or equivalently on the manifold $(M,g',e^{-(n-2)h}\mu')$ BE$(-K/T,N)$ is fulfilled with $K=4k_T(\meas)$. Hence we can apply~\cite{MR1450586}. Thus
\begin{equation*}
\overline{\lambda}_1 \ge \lambda(K/T),
\end{equation*}
where $\lambda(K/T)$ denotes the first mixed eigenvalue of the operator 
\begin{equation*}
\mathcal{L}:= 4 \frac{d^2}{dr^2} + J(r)\frac{d}{dr},\ J(r):=\frac{Kr}{T}
\end{equation*}
on $[0,D']$ with Neumann condition at $D'$ and Dirichlet condition at 0. Here $D'$ is the diameter of $(M,g',\mu')$.\\ 
It remains to bound $\lambda(K/T)$ below. From~\cite[Corollary 1]{MR1450586} we obtain
\begin{equation*}
\overline{\lambda}_1 \ge \frac{\pi^2 K}{8T(\exp(KD'^2/(8T))-1)}.
\end{equation*}
Further we have $h \le K$ and thus $D' \le e^{2K} D$ and
\begin{equation*}
\overline{\lambda}_1 \ge \frac{\pi^2 K}{8T(\exp(Ke^{4K}D^2/(8T))-1)}.
\end{equation*}
If $n=2$, then we only have a conformal change of metric. Then still following~\cite{MR1450586}
\begin{equation*}
\overline{\lambda}_1 \ge \tilde{\lambda}(K/T),
\end{equation*}
where $\tilde{\lambda}(K/T)$ denotes the first mixed eigenvalue of the operator 
\begin{equation*}
\tilde{\mathcal{L}}:= 4 \frac{d^2}{dr^2} + I(r)\frac{d}{dr},\ I(r):=2\sqrt{\frac{K}{T(n-1)}}\tanh\lp\frac{r}{2}\sqrt{\frac{K}{T(n-1)}}\rp
\end{equation*}
on $[0,D']$ with Neumann condition at $D'$ and Dirichlet condition at 0. And further
\begin{align*}
\tilde{\lambda}(K/T) &\ge \frac{\pi^2}{D'^2} \sqrt{\frac{1+2D'^2K}{\pi^4 T}} \cosh^{-1}\lp\frac{D'}{2}\sqrt{\frac{K}{T}}\rp \\
&\ge \frac{\pi^2}{e^{4K}D^2} \sqrt{\frac{1+2e^{4K}D^2 K}{\pi^4 T}} \cosh^{-1}\lp\frac{e^{4K}D}{2}\sqrt{\frac{K}{T}}\rp.
\end{align*}
\end{proof}

\begin{remark}
Note that in the previous proof we only use BE$(-K/T,\infty)$ as opposed to the stronger BE$(-K/T,N)$ when we bound $\overline{\lambda}_1$ from below by $\lambda(K/T)$. This is less precise but results in a simpler bound. Alternatively, according to~\cite[Theorem 14]{MR1789850}, we could bound $\overline{\lambda}_1$ from below by the first nontrivial eigenvalue in the following one-dimensional eigenvalue problem
\begin{gather*}
v''(x) - S(x)v'(x) = - \lambda v(x) \text{ on } \lp-\frac{D'}{2},\frac{D'}{2}\rp \text{ with } v'\lp-\frac{D'}{2}\rp = v'\lp\frac{D'}{2}\rp =0\\
\text{and } S(x):= - \sqrt{(N-1)K/T} \tanh\lp \sqrt{\frac{K}{T(N-1)}} x\rp.
\end{gather*} 
\end{remark}

\begin{remark}
One could obtain similar bounds on the spectral gap of the Neumann $p-$Laplacian on $(M,g,\mu)$ by comparison with the spectral gap of the drift Neumann $p-$Laplacian on the time changed manifold and by using upper and lower bounds on $h$ and existing bounds for the spectral gap of the p-Laplacian on manifolds with convex boundary under the BE$(K,N)$ condition as in~\cite{MR4293945,MR4471716}.
\end{remark}

\subsection{Logarithmic Sobolev inequality}

In this section, we consider the logarithmic Sobolev inequality on $M$. Writing $\mu$ for the normalised volume measure on $M$, this is
\begin{equation*}
\mathrm{Ent}_\mu\lp f^2\rp:= \int_M f^2 \ln\lp f^2\rp d\mu - \int_M f^2 d\mu \ln\lp\int_M f^2 d\mu\rp \le L \cdot \int_M |\nabla f|^2 d\mu \ \forall f\in H^1(M,g).
\end{equation*}
We aim here at bounding the logarithmic Sobolev constant $L$ (i.e.\ the optimal constant fulfilling the above inequality) on $M$ from above in terms of the contractive Dynkin
bounds, diameter and dimension.

Again bounding the logarithmic Sobolev constant from above requires control over both the geometry in the interior and of the boundary in general.
Explicit upper bounds for the (drift) Neumann Laplacian on compact manifolds with convex boundary and uniform lower bounds on the Ricci curvature have been studied in~\cite{MR1426537,MR1452551,MR1698947} while the case of nonconvex boundary but second fundamental form uniformly bounded from below is treated in~\cite{MR2344874}. We are not aware of results on upper bounds for the logarithmic Sobolev constant in the case of $L^p$ or Dynkin bounds on the Ricci curvature or second fundamental form. 

As for the Poincar\'{e} constant in our setting the time change constructed in Theorem~\ref{thm:timechange} allows to compare $L$ with the logarithmic Sobolev constant on the time-changed manifold, thus reducing the discussion to a setting in which upper bounds on the logarithmic Sobolev constant have been studied before (see above). The ansatz is thus again analogous to~\cite{MR2344874}.

\begin{theorem}\label{thm:logSobolev}
Let $(M^n,g)$ be a compact smooth Riemannian manifold of diameter $D$ with a non-empty boundary. Assume there exists $T>0$ such that \eqref{eq:D'} holds. Then for $n\ge2$ it holds
\begin{equation*}
 L \le \frac{4e^{4K}D^2}{\sqrt{(1+e^{4K}D^2 K/T)^2+4e^{4K}D^2 \overline{\lambda}_1} -1 -e^{4K}D^2K/T} \, \cdot 
\end{equation*}
This can be further bounded by replacing $\overline{\lambda}_1$ by $\lambda(K/T)$ or $\tilde{\lambda}(K/T)$ (for n=2).\\
Furthermore for $n=2$ it holds
\begin{equation*}
L\le 8T \frac{(\exp(6Ke^{4K}D^2/T)-1)\exp(1+2e^{4K}D^2K/T)}{K} \, \cdot
\end{equation*}
Here $K=4k_T(\meas)$ is the constant from Theorem~\ref{thm:timechange} and $\overline{\lambda}_1$ denotes the first nontrivial eigenvalue of the Neumann Laplacian on $(M,e^{2h}g,e^{2h}\mu_g)$.
\end{theorem}

\begin{proof}
Let $L(D,K,n)$ be the biggest logarithmic Sobolev constant reached on an $n$-dimensional compact Riemannian manifold with convex boundary, diameter $D$ and with respect to a measure $\frac{e^V d\mu}{\int_M e^V d\mu}$ such that BE$(-K,\infty)$ is fulfilled.\\ 
Now since $h$ was constructed such that $(M,e^{2h}g,e^{2h}\mu)$ (or equivalently $(M,g',e^{-(n-2)h}\mu')$) has convex boundary and fulfils BE$(-K/T,n)$ we have
\begin{equation*}
\mathrm{Ent}_{\overline{\mu}/\overline{Z}}(f^2) \le L(D',K/T,n) \cdot \int_M |\nabla' f|'^2 d\frac{\overline{\mu}}{\overline{Z}} = L(D',K/T,n) \cdot \int_M |\nabla f|^2 d\frac{\mu}{\overline{Z}}.
\end{equation*}
Furthermore recall that
\begin{equation*}\label{eq:entropy}
\mathrm{Ent}_\mu(f^2)= \inf_{a\in[0,\infty)} \int_M f^2\ln(f^2/a) - f^2 +a d\mu.
\end{equation*}
Therefore
\begin{equation*}
\mathrm{Ent}_{\overline{\mu}/\overline{Z}}(f^2) \ge \frac{1}{\overline{Z}} \inf_M e^{2h} \mathrm{Ent}_\mu(f^2).
\end{equation*}
Thus together
\begin{equation*}
    \mathrm{Ent}_\mu(f^2) \le \frac{L(D',K/T,n)}{\inf_M e^{2h}} \int_M |\nabla f|^2 d\mu.
\end{equation*}
Again without further assumptions on $\mathrm{Ric}_-$ and $\mathrm{I\!I}_-$ we know from the construction of $h$ in the proof of Theorem~\ref{thm:timechange} that $e^{2h}\ge1$. It remains to bound $L(D',K/T,n)$.
For $n\ge2$ by~\cite[Theorem 3.1]{MR1698947} we obtain
\begin{align*}
    L &\le \frac{4D'^2}{\sqrt{(1+D'^2 K/T)^2+4D'^2 \overline{\lambda}_1} -1 -D'^2K/T}\\
    &\le \frac{4e^{4K}D^2}{\sqrt{(1+e^{4K}D^2 K/T)^2+4e^{4K}D^2 \overline{\lambda}_1} -1 -e^{4K}D^2K/T}.
\end{align*}
Note that this bound for the logarithmic Sobolev constant from~\cite{MR1698947} features the spectral gap $\overline{\lambda}_1$ for the drift Laplacian on $(M,g',e^{-(n-2)h}\mu')$ or in other words the spectral gap for the Laplacian on $(M,e^{2h}g,e^{2h}\mu)$. As in the proof of Theorem~\ref{thm:spectralgap} $\overline{\lambda}_1$ can be further bounded in terms of $\lambda(K/T)$ or $\tilde{\lambda}(K/T)$ (for $n=2$).\\
For $n=2$ by~\cite[Theorem 3.3]{MR1452551} we obtain
\begin{align*}
    L&\le 8T \frac{(\exp(6KD'^2/T)-1)\exp(1+2D'^2K/T)}{K}\\
    &\le 8T \frac{(\exp(6Ke^{4K}D^2/T)-1)\exp(1+2e^{4K}D^2K/T)}{K}.
\end{align*}
\end{proof}

\bibliographystyle{abbrv}
\bibliography{bib}
\end{document}